\documentclass[12pt]{article}
\usepackage{amssymb}
\usepackage{latexsym}
\usepackage{amsmath}
\usepackage{appendix}

\newtheorem{theo}{Theorem}[section]

\newtheorem{ques}{Question}
\newtheorem{lemma}[theo]{Lemma}
\newtheorem{claim}[theo]{Claim}
\newtheorem{coro}[theo]{Corollary}

\def\q{\hspace*{\fill}$\Box$\medskip}
\begin{document}
\title{Intersecting families with full difference sets}
\author{Zilong Yan \thanks{School of Mathematics, Hunan University, Changsha 410082, P.R. China. Email: zilongyan@hnu.edu.cn. Supported in part by Postdoctoral fellow fund in China (No. 2023M741131)  and  Postdoctoral Fellowship Program of
CPSF under Grant (No. GZC20240455).} \and Yuejian Peng \thanks{ Corresponding author. School of Mathematics, Hunan University, Changsha, 410082, P.R. China. Email: ypeng1@hnu.edu.cn. \ Supported in part by National Natural Science Foundation of China (Nos. 11931002 and 12371327).}
	}
	
	\maketitle

\begin{abstract}
For a family $\mathcal{F}$ of subsets of a finite set, define $\mathcal{D}(\mathcal{F})=\{F\setminus F': F, F'\in\mathcal{F}\}$.  A family $\mathcal{F}$ is called intersecting if $F\cap F'\not=\emptyset$ for all $F, F'\in\mathcal{F}$. Frankl \cite{Frankl} showed that for a $k$-uniform intersecting family $\mathcal{F}\subset{[n]\choose k}$ with $n\ge k(k+3)$,  $|\mathcal{D}(\mathcal{F})|$ reaches the maximum  if and only if $\mathcal{F}$ is a  $k$-uniform full star. Later, Frankl-Kiselev-Kupavskii \cite{FKK} improved the bound $n\ge k(k+3)$ in the above result of Frankl \cite{Frankl}  to $n\ge 50klnk$ for  $k\ge 50$. For $2k<n<4k$, Frankl-Kiselev-Kupavskii \cite{FKK} showed that there exists a $k$-uniform family $\mathcal{F}\subset{[n]\choose k}$ such that $|\mathcal{D}(\mathcal{F})|$ is larger than $|\mathcal{D}(\mathcal{S})|$, where $\mathcal{S}$ is a full star. This result left the case $n=2k$ open and we show that in this case, $\mathcal{D}(\mathcal{F})$ can be `full' for some $\mathcal{F}\subset{[2k]\choose k}$.
It is clear that for an intersecting family $\mathcal{F}\subset{[n]\choose k}$, $\mathcal{D}(\mathcal{F})\subseteq \cup_{j=0}^{k-1}{[n]\choose j}$.   We say that a  $k$-uniform intersecting family $\mathcal{F}\subset{[n]\choose k}$ has full differences if $\mathcal{D}(\mathcal{F})=\cup_{j=0}^{k-1}{[n]\choose j}$, and we say that a  $k$-uniform intersecting family $\mathcal{F}\subset{[n]\choose k}$ has full differences of size $j$ if ${[n]\choose j} \subseteq \mathcal{D}(\mathcal{F})$. For odd $k$, a $k$-uniform intersecting family $\mathcal{F}\subset{[2k]\choose k}$ having full differences is given in \cite{Frankl}, and Frankl \cite{Frankl} asked for even $k\ge 4$ whether there exists a $k$-uniform intersecting family $\mathcal{F}\subset{[2k]\choose k}$ having full differences of size $k-1$. We answer this question in a stronger form and show that  for even $k\ge 4$, there exists  a $k$-uniform intersecting family $\mathcal{F}\subset{[2k]\choose k}$ having full differences. This also extends the range $2k<n<4k$ in the result of Frankl-Kiselev-Kupavskii \cite{FKK} to $2k\le n<4k$. Furthermore,  we show that if $\mathcal{F}\subset{[n]\choose k}$ is an intersecting family  having full differences of size $k-1$, then $n\le 3k-2$. The bound $n\le 3k-2$ is tight, as evidenced by the Fano plane.

	\end{abstract}

	\section{Introduction}
For a family $\mathcal{F}$ of subsets of a finite set, define the family of difference sets $\mathcal{D}(\mathcal{F})=\{F\setminus F': F, F'\in\mathcal{F}\}$. A natural question is how large $\vert\mathcal{D}(\mathcal{F})\vert$ can be. Marica-Sch\"onheim \cite{MaSc} showed that $|\mathcal{D}(\mathcal{F})|\ge |\mathcal{F}|$ holds for any $\mathcal{F}$.  This lower bound is reachable,  on the other hand, $|\mathcal{D}(\mathcal{F})|-|\mathcal{F}|$  can be `very large', so we may want to put some structures on $\mathcal{F}$ when we study $\mathcal{D}(\mathcal{F})$. One interesting class of families in extremal set theory is intersecting families. Various studies on intersection families (see \cite{FT16}) inspired by the foundational  result of Erd\H{o}s, Ko and Rado \cite{EKR1961} were given,  and some of them were applied in other fields. For example, the Ahlswede-Khachatrian theorm \cite{AK} which determines the maximum size of $\mathcal{A}\subseteq {[n]\choose k}$ in which any two sets intersect in at least $t$ elements, plays a crucial role in the hardness-of-approximation theorem for the `vertex cover' problem, proved by Dinur and Safra \cite{DS} (see \cite{KL21}).

A family $\mathcal{F}$ is called {\em intersecting} if $F\cap F'\not=\emptyset$ for all $F, F'\in\mathcal{F}$. Two families $\mathcal{F}$ and $\mathcal{H}$ are called {\em cross-intersecting} if $F\cap H\not=\emptyset$ for all $F\in\mathcal{F}$ and $H\in\mathcal{H}$. 
	For an intersecting family $\mathcal{F}$ of subsets of $[n]=\{1, 2, \ldots, n\}$, Frankl \cite{Frankl} showed that $|\mathcal{D}(\mathcal{F})|\le 2^{n-1}$. For a set $A$ and  a non-negative integer $k$, denote ${A\choose k}=\{F: F\subset A \ \text {and} \ |F\cap A|=k\}$. Frankl \cite{Frankl} also showed that for  an intersecting $k$-uniform family $\mathcal{F}\subset{[n]\choose k}$ with $n\ge k(k+3)$, we have $|\mathcal{D}(\mathcal{F})|\le {n-1\choose k-1}+{n-1\choose k-2}+\dots+{n-1\choose 0}$, equality attained only for a $k$-uniform full star on $[n]$ (consists of all $k$-subsets of $[n]$ containing a fixed element). Later, Frankl-Kiselev-Kupavskii \cite{FKK} improved the bound $n\ge k(k+3)$ in the above result of Frankl \cite{Frankl}  to
	$n\ge 50klnk$ for  $k\ge 50$. Frankl \cite{Frankl} asked whether $|\mathcal{D}(\mathcal{F})|$ reaches the maximum value on $k$-uniform full stars among all $k$-uniform intersecting families $\mathcal{F}$. However,
	Frankl-Kiselev-Kupavskii \cite{FKK} showed the following result.

\begin{theo}[Frankl-Kiselev-Kupavskii \cite{FKK}] For integers $n, k$ satisfying $2k<n<4k$ and $k$ sufficiently large, there is an  intersecting $k$-uniform family $\mathcal{F}\subset{[n]\choose k}$ such that  $|\mathcal{D}(\mathcal{F})|> |\mathcal{D}(\mathcal{S})|$, where $\mathcal{S}$ is a $k$-uniform full star on $[n]$.
\end{theo}

 This result left the case $n=2k$ open. For this case, we show that $\mathcal{D}(\mathcal{F})$ can be `full' for some $k$-uniform family $\mathcal{F}$. Let us be precise.  It is clear that for an intersecting family $\mathcal{F}\subset{[n]\choose k}$, $\mathcal{D}(\mathcal{F})\subseteq \cup_{j=0}^{k-1}{[n]\choose j}$.  We say that a  $k$-uniform intersecting family $\mathcal{F}\subset{[n]\choose k}$ has full differences if $\mathcal{D}(\mathcal{F})=\cup_{j=0}^{k-1}{[n]\choose j}$, and we say that a  $k$-uniform intersecting family $\mathcal{F}\subset{[n]\choose k}$ has full differences of size $j$ if ${[n]\choose j} \subseteq \mathcal{D}(\mathcal{F})$. In this paper, we study that for what values of $n$, there exists an intersecting family $\mathcal{F}\subset{[n]\choose k}$  having full differences. We think that this question has its own interests. It is equivalent to ask when there is an  intersecting family $\mathcal{F}\subset {[n]\choose k}$ such that every non-empty subset $S$ of $[n]$ with size no more than $k-1$  is contained in a sunflower of  $\mathcal{F}$ with $2$ petals in which $S$ is a petal.  A sunflower with $t$ petals is a family of  sets $\{P_1, P_2, \cdots, P_t\}$ satisfying that there is a set $C$  such that $P_i\cap P_j=C$ for any  $1\le i<j\le t$. We call  $C$ the kernel of the sunflower and call $P_i\setminus C$, $1\le i\le t$ the petals of the sunflower. A nonempty set $S\in \mathcal{D}(\mathcal{F})$ is equivalent to that $S$ is contained in a sunflower of  $\mathcal{F}$ with $2$ petals  in which $S$ is a petal. From this point of view, the existence of a $k$-uniform family with full differences is also related to combinatorial design. Constructing designs satisfying certain conditions is  very challenging in general.

For odd $k\ge 3$, an intersecting family $\mathcal{F}\subset{[2k]\choose k}$ having full differences  is given in \cite{Frankl}.

\begin{theo}[\cite{Frankl}]\label{oddk}
Let $k\ge 3$ be odd. Let $\mathcal{G}\subset{[2k]\choose k}$ consist of all $k$-sets $\{a_1, a_2, \dots, a_k\}$ such that $a_1+a_2+\dots+a_k$ is odd. Then $\mathcal{G}$ has full differences, i.e., $\mathcal{D}(\mathcal{G})=\cup_{j=0}^{k-1}{[2k]\choose j}$.
\end{theo}

In the same paper, Frankl \cite{Frankl} asked whether the same is true for even $k\ge 4$.
	\begin{ques}[Frankl \cite{Frankl}]\label{queseven}
		Suppose that $k\ge 4$ is even. Does there exist an intersecting family $\mathcal{F}\subset{[2k]\choose k}$ with full differences of size $k-1$ (i.e., ${[2k]\choose k-1}\subset \mathcal{D}(\mathcal{F})$)?
	\end{ques}
 It is not hard to see  that for $k=2$, there is no intersecting  $2$-uniform family on more than $3$ elements having full differences.
	In this paper, we answer Question \ref{queseven} and  show the following stronger result in Section \ref{2}.
	\begin{theo}\label{main}
		For even $k\ge 4$, there exists an intersecting family $\mathcal{F}\subset{[2k]\choose k}$ with full differences, i.e.,  $\mathcal{D}(\mathcal{F})=\cup_{j=0}^{k-1}{[2k]\choose j}$.
	\end{theo}

Combining  Theorems \ref{oddk} and  \ref{main},  there exists an intersecting family $\mathcal{F}\subset{[2k]\choose k}$ with full differences. 
It is natural  to study what values of $n>2k$ are such that there is an intersecting family $\mathcal{F}\subset{[n]\choose k}$ with full differences.
We will prove that if such an $n$ exists, it cannot be too large. Indeed, we will prove the following tight bound in Section \ref{sub3k-2}.
	
\begin{theo}\label{3k-2}
If $\mathcal{F}\subset{[n]\choose k}$ is an intersecting family satisfying ${[n]\choose k-1}\subset \mathcal{D}(\mathcal{F})$, then $n\le 3k-2$.
\end{theo}
	
 Note that when $k=3$ and $n=3k-2=7$, the Fano plane $\mathcal{F}=\{123, 146, 157, 247, 367,\\ 265, 345\}$ on $[7]$ satisfies ${[7] \choose 2}\subset \mathcal{D}(\mathcal{F})$. Thus, the bound $n\le 3k-2$ in the above theorem is tight. 
	
We will prove  Theorems  \ref{main} and \ref{3k-2} in  the next section.
	
\section{Families with full differences}\label{2}

Let $k\ge 4$ be even. Throughout this section, let
 $$\mathcal{N}=\{0, 1, 2, \dots, 2k-1\}, $$
  $$A=\{x\in \mathcal{N}: x\equiv0 \text{(mod 3)}\},$$
	$$B=\{x\in \mathcal{N}: x\equiv1 \text{(mod 3)}\},$$
	and
	$$C=\{x\in \mathcal{N}: x\equiv2 \text{(mod 3)}\}.$$
	For $0\le i\le 2,$ let
	$$\mathcal{R}_i=\{\{x_1, x_2, \dots, x_k\}\in{\mathcal{N}\choose k}: \sum_{j=1}^k x_j\equiv i(mod \ 3)\}.$$

\begin{lemma}\label{claimR1} Let $\mathcal{R}_i$ be defined as above. Then

	(i) If $k\equiv 0$ (mod 3) or $k\equiv 2$ (mod 3), then	$\mathcal{R}_1$ is intersecting, and $\mathcal{R}_0$ and $\mathcal{R}_1$ are cross-intersecting.

(ii)If $k\equiv 1$ (mod 3), then $\mathcal{R}_0$ is intersecting, and $\mathcal{R}_0$ and $\mathcal{R}_2$ are cross-intersecting.
	\end{lemma}
	\noindent {\em Proof of Lemma \ref{claimR1}.} We only prove (i) and the proof of (ii) is very similar.  Note that $\sum_{t\in \mathcal{N}} t\equiv 0$(mod 3) if $k\equiv 0$ (mod 3) or $k\equiv 2$ (mod 3). Suppose on  the contrary that there exist $T_1\in \mathcal{R}_1$ and $T_2\in \mathcal{R}_1$ such that $T_1\cap T_2=\emptyset$, then $T_1\cup T_2=\mathcal{N}$ and $$\sum_{i\in\mathcal{N}}i=\sum_{i\in T_1\cup T_2}i=\sum_{i\in T_1}i+\sum_{i\in T_2}i\equiv 2\text{(mod 3)},$$ contradict to $\sum_{i\in\mathcal{N}}i\equiv 0 \text{(mod 3)}$. Therefore $\mathcal{R}_1$ is intersecting.
	
	Let $R_0\in \mathcal{R}_0$ and $R_1\in \mathcal{R}_1$ be such that $R_0\cap R_1=\emptyset$, then $R_0\cup R_1=\mathcal{N}$ and $$\sum_{i\in\mathcal{N}}i=\sum_{i\in R_0\cup R_1}i=\sum_{i\in R_0}i+\sum_{i\in R_1}i\equiv 1\text{(mod 3)},$$ contradict to $\sum_{i\in\mathcal{N}}i\equiv 0 \text{(mod 3)}$. Therefore $\mathcal{R}_0$ and $\mathcal{R}_1$ are cross-intersecting. \q

\subsection{Preliminary results and sketch of the proof} \label{prel}

To show  Theorem  \ref{main}, we need to construct $\mathcal{F}\subset{\mathcal{N}\choose k}$  such that $\mathcal{F}$ is intersecting and ${\mathcal{N}\choose i}\subset\mathcal{D}(\mathcal{F})$ for $1\le i\le k-1$. We need to prescribe what kind of $k$-subsets of $\mathcal{N}$ to be included in $\mathcal{F}$ such that $\mathcal{F}$ is intersecting and $\cup_{j=1}^{k-1}{\mathcal{N}\choose j}\subset \mathcal{D}(\mathcal{F})$.  We first focus on constructing an intersecting family satisfying the requirement ${\mathcal{N}\choose k-1}\subset \mathcal{D}(\mathcal{F})$. Then we show that our construction satisfies $\cup_{j=1}^{k-2}{\mathcal{N}\choose j}\subset \mathcal{D}(\mathcal{F})$ as well. Take any $X\in {\mathcal{N}\choose k-1}$. 
  Our aim is to construct intersecting $\mathcal{F}\subseteq {\mathcal{N}\choose k}$ such that $X\in \mathcal{D}(\mathcal{F})$ for all $X\in{\mathcal{N}\choose k-1}$. Let $Y=\mathcal{N}\setminus X$.
The following lemma gives us some idea for sets to be included in $\mathcal{F}$ such that all $(k-1)$-sets not containing any of $A, B$ or $C$ are in $\mathcal{D}(\mathcal{F})$. We will show the following lemma first.
\begin{lemma}\label{normal}
Let $X$ be a $(k-1)$-subset of $\mathcal{N}$ satisfying $A\not\subset X$, $B\not\subset X$ and $C\not\subset X$. If $(\sum_{t\in \mathcal{N}}t)$ mod 3$=j$, then $X\in \mathcal{D}(\mathcal{R}_i)$ for $i= (2-j)$ mod 3 and $i= (4-j)$ mod 3.
	\end{lemma}
	\noindent {\em Proof of Lemma \ref{normal}.} We aim to find $X', Y'\in \mathcal{R}_i$ such that $X=X'\setminus Y'$ for $i=(2-j)$mod 3 and $i=(1-j)$mod 3. Let $Y=\mathcal{N}\setminus X$. Since $A\not\subset X$, $B\not\subset X$ and $C\not\subset X$, we have $Y\cap A\not=\emptyset$, $Y\cap B\not=\emptyset$ and $Y\cap C\not=\emptyset$.
	
Suppose that  $\sum_{x\in X} x\equiv l$(mod 3), then $\sum_{y\in Y} y\equiv (j-l)$(mod 3).
Let $h=(i+j)$mod 3. Let $y'\neq y''\in Y$ such that $y'\equiv (i-l)$(mod 3) and $y''\equiv (y'+h)$(mod 3). Define $X'=X\cup\{y'\}$ and $Y'=Y\setminus\{y''\}$.  Then $$\sum_{x\in X'}x=\sum_{x\in X}x+y' \ \text{and} \ \sum_{x\in X}x+y'\equiv(l+(i-l))(\text{mod} \ 3)\equiv i(\text{mod} \ 3),$$	
$$\sum_{y\in Y'}y=\sum_{y\in Y}y-y'',$$ and
$$\sum_{y\in Y}y-y''\equiv((j-l)-(y'+h))(\text{mod} \ 3)\equiv((j-l)-(i-l+h))(\text{mod} \ 3)\equiv i(\text{mod} \  3)$$
and $X=X'\setminus Y'$. Thus $X\in \mathcal{D}(\mathcal{R}_i)$.\q

Let us sketch the idea for the case $k\equiv 0$ (mod 3) very roughly. We first put all members in $\mathcal{R}_1$ to $\mathcal{F}$, i.e. $\mathcal{R}_1\subset \mathcal{F}$. By Lemma \ref{normal}, $X\in \mathcal{D}(\mathcal{F})$ if $A\not\subset X$, $B\not\subset X$ and $C\not\subset X$. For other cases, we will add an intersecting subfamily $
\mathcal{R}^*_0\subseteq\mathcal{R}_0$ into $\mathcal {F}$ such that $\mathcal{F}=\mathcal{R}^*_0\cup \mathcal{R}_1$ is intersecting and ${\mathcal{N}\choose k-1}\subset\mathcal{D}(\mathcal{F})$. By Lemma \ref{claimR1}, $\mathcal{F}$ is intersecting provided $\mathcal{R}^*_0$ is intersecting. In order to construct $\mathcal{R}^*_0$ satisfying the requirements, we try to add a proper element $v$ into $X$, and delete a proper element $u$ from $Y=\mathcal{N}\setminus X$ such that $X=(X\cup\{v\})\setminus(Y\setminus\{u\})$. We need to guarantee that the family $\mathcal{R}^*_0$, formed by all such sets $X\cup\{v\}$ and $Y\setminus\{u\}$, is intersecting.
 For this purpose, we prove  three preliminary lemmas (Lemmas \ref{sum}, \ref{add} and \ref{delete}) to guide us how to add a proper element $v$ into $X$, and delete a proper element $u$ from $Y=\mathcal{N}\setminus X$. Before stating the lemmas, we need some definitions.

For a graph $G$, let $V(G)$,  $E(G)$ and $e(G)$ denote its vertex set, edge set and  the number of edges respectively. For a set $U\subset V(G)$, let $N(U)=\{v: v \ \text {is\ adjacent\ to\ some\ vertex}\ \\  u\in U\}$. For disjoint $U, V\subseteq V(G)$, $G[U, V]$ denotes the bipartite subgraph of $G$ induced by $U$ and $V$, i.e. $G[U,V]$ consists of all edges incident to one vertex in $U$ and one vertex in $V$.

Let $\mathcal{Z}_1=\mathcal{Z}_1(U_1, V_1)$ be the bipartite graph with $U_1=\{u_1, u_2, \dots, u_s\}$, $V_1=\{v_1, v_2, \dots, \\ v_s\}$, and $E(\mathcal{Z}_1)=\{u_iv_i: 1\le i\le s\}$. Let $\mathcal{Z}_2=\mathcal{Z}_2(U_2, V_2)$ be the bipartite graph with $U_2=\{u_1, u_2, \dots, u_s, u_{s+1}\}$, $V_2=\{v_1, v_2, \dots, v_s\}$, and $E(\mathcal{Z}_2)=\{u_iv_i: 1\le i\le s\}$. Let $\mathcal{Z}_3=\mathcal{Z}_3(U_3, V_3)$ be the bipartite graph with $U_3=\{u_1, u_2, \dots, u_s\}$, $V_3=\{v_1, v_2, \dots, v_s, v_{s+1}\}$, and $E(\mathcal{Z}_3)=\{u_iv_i: 1\le i\le s\}$.

\begin{lemma}\label{sum}
Let $i\in\{1, 2, 3\}$. Let $U_i=U\cup U'$ and $V_i=V\cup V'$ be partitions of $U_i$ and $V_i$ respectively. If $|U|+|V|-(|U\cap\{u_{s+1}\}|+|V\cap\{v_{s+1}\}|)\equiv 0$(mod 2), then $e(\mathcal{Z}_i[U, V])+e(\mathcal{Z}_i[U', V'])\equiv s$ (mod 2). If $|U|+|V|-(|U\cap\{u_{s+1}\}|+|V\cap\{v_{s+1}\}|)\equiv 1$(mod 2), then $e(\mathcal{Z}_i[U, V])+e(\mathcal{Z}_i[U', V'])\equiv (s+1)$ (mod 2).
	\end{lemma}
	\noindent {\em Proof of Lemma \ref{sum}.} We prove for $\mathcal{Z}_1$ first. Let $e(\mathcal{Z}_1[U, V])=x$, and $e(\mathcal{Z}_1[U', V'])=y$. Let $n_1=|\{t: t\in U\cup V, N(t)\notin U\cup V\}|$, then $n_1+2x=|U|+|V|$ and $s=x+y+n_1$. If $|U|+|V|\equiv 0$(mod 2), then $n_1\equiv 0$(mod 2). Therefore $s\equiv (x+y)$(mod 2). If $|U|+|V|\equiv 1$(mod 2), then $n_1\equiv 1$(mod 2). Therefore $s+1\equiv (x+y)$(mod 2).
	
For $\mathcal{Z}_2$, and $u_{s+1}\notin U$, it follows from $\mathcal{Z}_1$. If $u_{s+1}\in U$, since we have shown for $(\mathcal{Z}_1, U\setminus\{u_{s+1}\}, V)$, it holds for $(\mathcal{Z}_2, U, V)$.
	
For $\mathcal{Z}_3$, and $v_{s+1}\notin U$, it follows from $\mathcal{Z}_3$. If $v_{s+1}\in V$, since we have shown for $(\mathcal{Z}_1, U, V\setminus\{v_{s+1}\})$, it holds for $(\mathcal{Z}_3, U, V)$.
	\q
	
	\begin{lemma}\label{add}
		Let $i\in\{1, 2, 3\}$. Let $U\subset U_i$ and $V\subset V_i$ with $|U|+|V|\le s-1$. If $|U|-|U\cap\{u_{s+1}\}|\le |V|-|V\cap\{v_{s+1}\}|$, then there exists $v\in (V_i\setminus V)\setminus\{v_{s+1}\}$ such that $e(\mathcal{Z}_i[U, V\cup\{v\}])\equiv (|U|-|U\cap\{u_{s+1}\}|)$(mod 2). If $|U|-|U\cap\{u_{s+1}\}|>|V|-|V\cap\{v_{s+1}\}|$, then there exists $v\in (V_i\setminus V)\setminus\{v_{s+1}\}$ such that $e(\mathcal{Z}_i[U, V\cup\{v\}])\equiv (|V|-|V\cap\{v_{s+1}\}|)$(mod 2). Furthermore $N(V\cup\{v\})\not\subset U$.
	\end{lemma}
	\noindent {\em Proof of Lemma \ref{add}.} We prove it for $\mathcal{Z}_1$ first. Since $|U|+|V|\le s-1$, we have $|U|<|V_1\setminus V|$. So there exists $v'\in V_1\setminus V$ such that $N(v')\notin U$. We show the situation $|U|\le |V|$ first. If $e(\mathcal{Z}_1[U, V])\equiv |U|$(mod 2), then let $v=v'$. If $e(\mathcal{Z}_1[U, V])\not\equiv |U|$(mod 2), then there exists $u\in U$ such that $N(u)\notin V$. Let $v=N(u)$, then we are done. Next we show the situation $|U|>|V|$. Note that there exists $u\in U$ such that $N(u)\notin V$. If $e(\mathcal{Z}_1[U, V])\not\equiv |V|$(mod 2), then let $v=N(u)$. If $e(\mathcal{Z}_1[U, V])\equiv |V|$(mod 2), then let $v=v'$.
	
	For $\mathcal{Z}_2$, and $u_{s+1}\notin U$, it follows from $\mathcal{Z}_1$. If $u_{s+1}\in U$, since we have shown for $(\mathcal{Z}_1, U\setminus\{u_{s+1}\}, V)$, it holds for $(\mathcal{Z}_2, U, V)$.
	
	For $\mathcal{Z}_3$, and $v_{s+1}\notin V$, it follows from $\mathcal{Z}_1$. If $v_{s+1}\in V$, since we have shown for $(\mathcal{Z}_1, U, V\setminus\{v_{s+1}\})$, it holds for $(\mathcal{Z}_3, U, V)$.
	\q
	
	\begin{lemma}\label{delete}
	Let $i\in\{1, 2, 3\}$.	Let $U\subset U_i$ and $V\subset V_i$ with $|U|+|V|\ge s+2$ and $N(U)\not\subset V$. Then there exists $u\in U\setminus\{u_{s+1}\}$ such that $e(\mathcal{Z}_i[U\setminus\{u\}, V])\equiv 0$(mod 2) and there exists $u'\in U\setminus\{u_{s+1}\}$ such that $e(\mathcal{Z}_i[U\setminus\{u'\}, V])\equiv 1$(mod 2).
	\end{lemma}
	\noindent {\em Proof of Lemma \ref{delete}.} Since $|U|+|V|\ge s+2$, there exists $u''\in U\setminus\{u_{s+1}\}$ such that $v''=N(u'')\in V$. Since $N(U)\not\subset V$, there exists $u'''\in U\setminus\{u_{s+1}\}$ such that $v'''=N(u''')\notin V\cup\{v_{s+1}\}$. Without loss of generality, we may assume that $e(\mathcal{Z}_i[U, V])\equiv 0$(mod 2). Therefore $e(\mathcal{Z}_i[U\setminus\{u'''\}, V])\equiv 0$(mod 2) and $e(\mathcal{Z}_i[U\setminus\{u''\}, V])\equiv 1$(mod 2).
\q

\subsection{Proof of Theorem \ref{main}}

	We will  apply Lemmas \ref{normal}, \ref{sum}, \ref{add} and \ref{delete} to construct $\mathcal{F}$ such that $\mathcal{F}$ is intersecting and $\cup_{j=1}^{k-1}{\mathcal{N}\choose j}\subset \mathcal{D}(\mathcal{F})$. There are three cases  according to $i=0, 1,$ or $2,$ where $i\equiv k$(mod 3). The proofs are similar for three cases with some modifications. 

{\bf Case 1.  $k\equiv 0$ (mod 3).}

	Let $k=3p$ be even. In this case,
$$\mathcal{N}=\{0, 1, 2, \dots, 6p-1\},$$
 $$A=\{0, 3, 6, \dots, 6p-3\},$$
  $$B=\{1, 4, 7, \dots, 6p-2\}, $$
   $$C=\{2, 5, 8, \dots, 6p-1\}.$$

    We first focus on constructing an intersecting family satisfying the requirement ${\mathcal{N}\choose 3p-1}\subset \mathcal{D}(\mathcal{F})$. Then we show that our construction satisfies $\cup_{j=1}^{3p-2}{\mathcal{N}\choose j}\subset \mathcal{D}(\mathcal{F})$ as well. Take any $X\in {\mathcal{N}\choose 3p-1}$, let
$$|X\cap A|=a, \ |X\cap B|=b \ {\text {and}} \ |X\cap C|=c,$$ then $a+b+c=k-1=3p-1$. Our aim is to construct intersecting $\mathcal{F}\subseteq {\mathcal{N}\choose 3p}$ such that $X\in \mathcal{D}(\mathcal{F})$ for all $X\in{\mathcal{N}\choose 3p-1}$. Let $Y=\mathcal{N}\setminus X$. Note that $\sum_{t\in \mathcal{N}} t\equiv 0$(mod 3) and $\sum_{x\in X} x\equiv b+2c$(mod 3). 

We first put all members in $\mathcal{R}_1$ to $\mathcal{F}$, i.e. $\mathcal{R}_1\subset \mathcal{F}$. By Lemma \ref{normal}, if $a, b, c<2p$, then $X\in \mathcal{D}(\mathcal{F})$. For the situation that $a=2p$ or $b=2p$ or $c=2p$, we will add an intersecting subfamily $
\mathcal{R}^*_0\subseteq\mathcal{R}_0$ into $\mathcal {F}$ such that $\mathcal{F}=\mathcal{R}^*_0\cup \mathcal{R}_1$ is intersecting and ${\mathcal{N}\choose 3p-1}\subset\mathcal{D}(\mathcal{F})$. By Lemma \ref{claimR1}, $\mathcal{F}$ is intersecting provided $\mathcal{R}^*_0$ is intersecting. In order to construct $\mathcal{R}^*_0$ satisfying the requirements, we try to add a proper element $v$ into $X$, and delete a proper element $u$ from $Y=\mathcal{N}\setminus X$ such that $X=(X\cup\{v\})\setminus(Y\setminus\{u\})$. We need to guarantee that the family $\mathcal{R}^*_0$, formed by all such sets $X\cup\{v\}$ and $Y\setminus\{u\}$, is intersecting. 

We construct an auxiliary tripartite graph $G$ with partition $A\cup B\cup C$ and let $\{\alpha, \alpha+1, \alpha+2\}$ form a $K_3$ for each $\alpha\in A$. Therefore $G[A\cup B], G[A\cup C]$ and $G[B\cup C]$ are perfect matchings. Note that $|A|=|B|=|C|$, $|A|\equiv 0$(mod 2) and $p\equiv 0$(mod 2).
	
 Let us consider the case $a=2p$ first. Note that $\sum_{x\in X} x\equiv b+2c$ (mod 3). If $b+2c\equiv 0$(mod 3), then $\sum_{y\in Y} y\equiv 0$(mod 3). Let $\beta\in Y\cap B$ and $\gamma\in Y\cap C$. Then $X\cup\{\beta\}, Y\setminus\{\gamma\}\in \mathcal{R}_1\subset \mathcal{F}$ and $X=(X\cup\{\beta\})\setminus(Y\setminus\{\gamma\})\in\mathcal{D}(\mathcal{F})$. In this case, we are fine.
	
	For the case $a=2p$ and $b+2c\equiv 1$(mod 3), applying Lemma \ref{add} by taking $s=2p$, $U_1=B$, $V_1=C$, $U=X\cap B$ and $V=X\cap C$, we have the following claim.
	
	\begin{claim}\label{case1.1add}
		Suppose that $a=2p$ and $b+2c\equiv 1$(mod 3). If $b\le c$, then there exists a vertex $\gamma'\in Y\cap C$ such that
		$e(G[X\cap B, (X\cap C)\cup\{\gamma'\}])\equiv b$(mod 2). If $b>c$, then there exists a vertex $\gamma'\in Y\cap C$ such that $e(G[X\cap B, (X\cap C)\cup\{\gamma'\}])\equiv c$(mod 2). Furthermore $N((X\cap C)\cup\{\gamma'\})\not\subset X\cap B.$
	\end{claim}
	
Let us continue the case $a=2p$ and $b+2c\equiv 1$(mod 3). Take $\gamma'$ as in Claim \ref{case1.1add}, and let ${\mathcal{R}_{0a}^1}$ be the family of all such $X\cup\{\gamma'\}$. Note that $X\cup\{\gamma'\}\in \mathcal{R}_{0}$, therefore $\mathcal{R}_{0a}^1\subset \mathcal{R}_{0}$. Take $\beta'\in Y\cap B$, since $\sum_{y\in Y} y\equiv 2$(mod 3), we have $Y\setminus\{\beta'\}\in \mathcal{R}_1$. Now $X=(X\cup\{\gamma'\})\setminus(Y\setminus\{\beta'\})\in \mathcal{D}(\mathcal{F}).$ Clearly $\mathcal{R}_{0a}^1$ is intersecting since each $k$-set in $\mathcal{R}_{0a}^1$ contains $A$.
	
For the case $a=2p$ and $b+2c\equiv 2$ (mod 3), applying Lemma \ref{delete} by taking $s=2p$, $U_1=B$, $V_1=C$, $U=Y\cap B$ and $V=Y\cap C$, we have the following claim.
	
\begin{claim}\label{case1.1delete}
Suppose that $a=2p$, $b+2c\equiv 2$(mod 3) and $N(Y\cap B)\not\subset Y\cap C$. If $b\le c-2$, then there exists $\beta''\in Y\cap B$ such that $e(G[(Y\cap B)\setminus\{\beta''\}, Y\cap C])\equiv b$(mod 2). If $b>c-2$, then there exists $\beta''\in Y\cap B$ such that $e(G[(Y\cap B)\setminus\{\beta''\}, Y\cap C])\equiv c$(mod 2).
\end{claim}
	
Let us continue the case $a=2p$ and $b+2c\equiv 2$(mod 3). Take $\gamma''\in Y\cap C$, note that $X\cup\{\gamma''\}\in\mathcal{R}_1$. If $N(Y\cap B)\not\subset Y\cap C$, then take $\beta''$ as in Claim \ref{case1.1delete}. If $N(Y\cap B)\subset Y\cap C$, then take any fixed $\beta''\in Y\cap B$. Let ${\mathcal{R}_{0a}^2}$ be the family of all such $Y\cup\{\beta''\}$, note that $Y\cup\{\beta''\}\in \mathcal{R}_{0}$, therefore $\mathcal{R}_{0a}^2\subset \mathcal{R}_{0}$.
So $X=(X\cup\{\gamma''\})\setminus(Y\setminus\{\beta''\})\in \mathcal{D}(\mathcal{F})$. Clearly $\mathcal{R}_{0a}^2$ is intersecting since for any $R_0^1$ and $R_0^2$ in $\mathcal{R}_{0a}^2$ we have $|R_0^1|=|R_0^2|=3p$ and $R_0^1\cup R_0^2\subset B\cup C$, then by the inclusion-exclusion principle $|R_0^1\cap R_0^2|=|R_0^1|+|R_0^2|-|R_0^1\cup R_0^2|\ge 2p$.
	
\begin{claim}\label{case1-r0aintersecting}
Let $\mathcal{R}_{0a}^*=\mathcal{R}_{0a}^1\cup\mathcal{R}_{0a}^2$, then $\mathcal{R}_{0a}^*$ is intersecting.
\end{claim}
\noindent {\em Proof of Claim \ref{case1-r0aintersecting}.} Since we have explained that $\mathcal{R}_{0a}^1$ is intersecting and $\mathcal{R}_{0a}^2$ is intersecting, it is sufficient to show that $\mathcal{R}_{0a}^1$ and $\mathcal{R}_{0a}^2$ are cross-intersecting. Suppose that there exist $R_1\in \mathcal{R}_{0a}^1$ and $R_2\in \mathcal{R}_{0a}^2$ such that $R_1\cap R_2=\emptyset.$ If $N(R_2\cap B)\subset R_2\cap C$, then $N(R_1\cap C)\subset R_1\cap B$, a contradiction to Claim \ref{case1.1add}. So $N(R_2\cap B)\not\subset R_2\cap C$. Let
\begin{align}\label{eq1}
|R_1\cap B|=b'\quad \text {and} \quad |R_1\cap C|-1=c',
\end{align}
then $b'+2c'\equiv 1$(mod 3) by the definition of $\mathcal{R}_{0a}^1$, $|R_2\cap B|=2p-|R_1\cap B|=2p-b'$ and $|R_2\cap C|=2p-|R_1\cap C|=2p-c'-1$.
By Claim \ref{case1.1add}, we have
\begin{align}\label{eq2}
e(G[R_1\cap B, R_1\cap C])\equiv b' \ \text {(mod 2)} \quad \text {if} \quad b'\le |R_1\cap C|-1=c',
\end{align}
and
\begin{align}\label{eq3}
e(G[R_1\cap B, R_1\cap C])\equiv |R_1\cap C|-1=c' \ \text {(mod 2)} \quad \text {if} \quad  b'>|R_1\cap C|-1=c'.
\end{align}
By Claim \ref{case1.1delete}, $e(G[R_2\cap B, R_2\cap C])\equiv (2p-|R_2\cap B|-1)(mod 2)=b'-1$(mod 2) if $2p-|R_2\cap B|-1\le 2p-|R_2\cap C|-2$, i.e.
\begin{align}\label{eq4}
e(G[R_2\cap B, R_2\cap C])\equiv (b'-1) \  \text {(mod 2)}\quad\text {if} \quad b'\le c',
\end{align}
and $e(G[R_2\cap B, R_2\cap C])\equiv 2p-|R_2\cap C|=c'+1$(mod 2) if $2p-|R_2\cap B|-1>2p-|R_2\cap C|-2$, i.e.
\begin{align}\label{eq5}
e(G[R_2\cap B, R_2\cap C])\equiv c'+1 \ \text{(mod 2)}\quad \text{if}\quad b'>c'.
\end{align}
Combining (\ref{eq2}) and (\ref{eq4}), (\ref{eq3}) and (\ref{eq5}), we have $e(G[R_1\cap B, R_1\cap C])+e(G[R_2\cap B, R_2\cap C])\equiv 1$(mod 2). Since $p$ is even, $|R_1\cap B|+|R_1\cap C|=p$ is even. By Lemma \ref{sum}, $e(G[R_1\cap B, R_1\cap C])+e(G[R_2\cap B, R_2\cap C])$ is even, a contradiction.\q

So far, we have shown that if $a, b, c<2p$ or $a=2p$, we have obtained an intersecting family $\mathcal{R}_{1}\cup\mathcal{R}_{0a}^*$ such that $X\in \mathcal{D}(\mathcal{R}_{1}\cup\mathcal{R}_{0a}^*)$. For the case $b=2p$ (or $c=2p$), we will construct $\mathcal{R}_{0b}^*$ (or $\mathcal{R}_{0c}^*$) such that $\mathcal{R}_{1}\cup\mathcal{R}_{0b}^*$ is intersecting (or $\mathcal{R}_{1}\cup\mathcal{R}_{0c}^*$ is intersecting) and $X\in \mathcal{D}(\mathcal{R}_{1}\cup\mathcal{R}_{0b}^*)$ (or $X\in \mathcal{D}(\mathcal{R}_{1}\cup\mathcal{R}_{0a}^*)$).

Now, we consider $b=2p$. Note that $\sum_{x\in X} x\equiv 2p+2c$ (mod 3). If $2p+2c\equiv 2$(mod 3), then $\sum_{y\in Y} y\equiv 1$(mod 3). Let $\alpha\in Y\cap A$ and $\gamma \in Y\cap C$. Then $X\cup\{\gamma\}, Y\setminus\{\alpha\}\in \mathcal{R}_1\subset \mathcal{F}$ and $X=(X\cup\{\gamma\})\setminus(Y\setminus\{\alpha\})\in\mathcal{D}(\mathcal{F})$. In this case, we are fine.
	
For the case $b=2p$ and $2p+2c\equiv 0$ (mod 3), applying Lemma \ref{add} by taking $s=2p$, $U_1=C$, $V_1=A$, $U=X\cap C$ and $V=X\cap A$ (note that $|U|+|V|=p\le s-1$), we have the following claim.
	
\begin{claim}\label{case1.2add}
Suppose that $b=2p$ and $2p+2c\equiv 0$(mod 3). If $c\le a$, then there exists a vertex $\alpha'\in Y\cap A$ such that $e(G[(X\cap A)\cup\{\alpha'\}, X\cap C])\equiv c$(mod 2). If $c>a$, then there exists a vertex $\alpha'\in Y\cap A$ such that $e(G[(X\cap A)\cup\{\alpha'\}, X\cap C])\equiv a$(mod 2). Furthermore $N((X\cap A)\cup\{\alpha'\})\not\subset X\cap C.$
\end{claim}
	
	Let us continue the case $b=2p$ and $2p+2c\equiv 0$(mod 3). Take $\alpha'$ as in Claim \ref{case1.2add} and let $\mathcal{R}_{0b}^1$ be the family consisting of all such $X\cup\{\alpha'\}$. Note that $\mathcal{R}_{0b}^1\subset\mathcal{R}_{0b}^*.$ Note that $\sum_{y\in Y} y\equiv 0$(mod 3), take $\gamma'\in Y\cap C$, then $Y\setminus\{\gamma'\}\in \mathcal{R}_1$ and $X=(X\cup\{\alpha'\})\setminus(Y\setminus\{\gamma'\})\in \mathcal{D}(\mathcal{F}).$ Clearly $\mathcal{R}_{0b}^1$ is intersecting since each $k$-set in $\mathcal{R}_{0b}^1$ contains $B$.
	
	For the case $b=2p$ and $2p+2c\equiv 1$ (mod 3), applying Lemma \ref{delete} by taking $s=2p$, $U_1=C$, $V_1=A$, $U=Y\cap C$ and $V=Y\cap A$, we have the following claim.
	
\begin{claim}\label{case1.2delete}
Suppose that $b=2p$, $2p+2c\equiv 1$(mod 3). Let $N(Y\cap C)\not\subset Y\cap A$. If $c\le a-2$, then there exists $\gamma''\in Y\cap C$ such that $e(G[Y\cap A, (Y\cap C)\setminus\{\gamma''\}])\equiv c$(mod 2). If $c>a-2$, then there exists $\gamma''\in Y\cap C$ such that $e(G[Y\cap A, (Y\cap C)\setminus\{\gamma''\}]\equiv a$(mod 2) .
\end{claim}
	
Let us continue the case $b=2p$ and $2p+2c\equiv 1$(mod 3). Note that for $\alpha''\in Y\cap A$ $X\cup\{\alpha''\}\in\mathcal{R}_1$. If $N(Y\cap C)\not\subset Y\cap A$, then take $\gamma''$ as in Claim \ref{case1.2delete}. If $N(Y\cap C)\subset Y\cap A$, then take any fixed $\gamma''$ in $Y\cap C$. Let $\mathcal{R}_{0b}^2$ be the family of all such $Y\setminus\{\gamma''\}$. Note that $\mathcal{R}_{0b}^2\subset\mathcal{R}_{0b}^*$ and $X=(X\cup\{\alpha''\})\setminus(Y\setminus\{\gamma''\})\in \mathcal{D}(\mathcal{F})$. Clearly $\mathcal{R}_{0b}^2$ is intersecting since for any $R_0^1$ and $R_0^2$ in $\mathcal{R}_{0b}^2$ we have $|R_0^1|=|R_0^2|=3p$ and $R_0^1\cup R_0^2\subset A\cup C$, then by the inclusion-exclusion principle $|R_0^1\cap R_0^2|=|R_0^1|+|R_0^2|-|R_0^1\cup R_0^2|\ge 2p$.
	
\begin{claim}\label{case1-r0bintersecting}
Let $\mathcal{R}_{0b}^*=\mathcal{R}_{0b}^1\cup\mathcal{R}_{0b}^2$, then $\mathcal{R}_{0b}^*$ is intersecting.
\end{claim}
\noindent {\em Proof of Claim \ref{case1-r0bintersecting}.} Since we have explained that $\mathcal{R}_{0b}^1$ is intersecting and $\mathcal{R}_{0b}^2$ is intersecting, it is sufficient to show that $\mathcal{R}_{0b}^1$ and $\mathcal{R}_{0b}^2$ are cross-intersecting. Suppose that there exist $R_1\in \mathcal{R}_{0b}^1$ and $R_2\in \mathcal{R}_{0b}^2$ such that $R_1\cap R_2=\emptyset$. If $N(R_2\cap C)\subset R_2\cap A$, then $N(R_1\cap A)\subset R_1\cap C$, a contradiction to Claim \ref{case1.2add}. So we may assume that $N(R_2\cap C)\not\subset R_2\cap A$.
	Let $c'=|R_1\cap C|$ and $a'=|R_1\cap A|-1$, then $|R_2\cap c|=2p-|R_1\cap C|=2p-c'$ and $|R_2\cap A|=2p-|R_1\cap A|=2p-a'-1$. By Claim \ref{case1.2add}, $e(G[R_1\cap A, R_1\cap C])\equiv |R_1\cap C|=c'$(mod 2) if $|R_1\cap C|\le |R_1\cap A|-1$, i.e. $e(G[R_1\cap A, R_1\cap C])\equiv c'$(mod 2) if $c'\le a'$,
	and $e(G[R_1\cap A, R_1\cap C])\equiv |R_1\cap A|-1=a'$(mod 2) if $|R_1\cap C|>|R_1\cap A|-1$, i.e. $e(G[R_1\cap A, R_1\cap C])\equiv a'$(mod 2) if $c'>a'$.
	By Claim \ref{case1.2delete}, $e(G[R_2\cap A, R_2\cap C])\equiv 2p-|R_2\cap C|-1=c'-1$(mod 2) if $2p-|R_2\cap C|-1\le 2p-|R_2\cap A|-2$, i.e. $e(G[R_2\cap A, R_2\cap C])\equiv c'-1$(mod 2) if $c'\le a'$,
	and $e(G[R_2\cap A, R_2\cap C])\equiv 2p-|R_2\cap A|=a'+1$(mod 2) if $2p-|R_2\cap C|-1>2p-|R_2\cap A|-2$, i.e. $e(G[R_2\cap A, R_2\cap C])\equiv a'+1$(mod 2) if $c'>a'$.
	So $e(G[R_1\cap A, R_1\cap C])+e(G[R_2\cap A, R_2\cap C])\equiv 1$(mod 2). Since $p$ is even, $|R_1\cap A|+|R_1\cap C|=p$ is even. By Lemma \ref{sum}, $e(G[R_1\cap A, R_1\cap C])+e(G[R_2\cap A, R_2\cap C])$ is even, a contradiction.\q
	
\begin{claim}\label{case1-r0ar0bintersecting}
$\mathcal{R}_{0a}^*$ and $\mathcal{R}_{0b}^*$ are cross-intersecting.
\end{claim}
\noindent {\em Proof of Claim \ref{case1-r0ar0bintersecting}.}Recall that $\mathcal{R}^*_{0a}=\mathcal{R}^1_{0a}\cup\mathcal{R}^2_{0a}$, $\mathcal{R}^*_{0b}=\mathcal{R}^1_{0b}\cup\mathcal{R}^2_{0b}$, and every member in $\mathcal{R}^1_{0a}$ contains all elements in $A$ and at least one element in $C$, every member in $\mathcal{R}^2_{0a}$ consists of some elements  but not all elements in $B$ and some elements in $C$, every member in $\mathcal{R}^1_{0b}$ contains all elements in $B$ and at least one element in $A$, every member in $\mathcal{R}^2_{0b}$ consists of some elements but not all elements in $C$ and some elements in $A$. Clearly, $\mathcal{R}^1_{0a}$ and $\mathcal{R}^1_{0b}$ are cross-intersecting, $\mathcal{R}^1_{0a}$ and $\mathcal{R}^2_{0b}$ are cross-intersecting and $\mathcal{R}^2_{0a}$ and $\mathcal{R}^1_{0b}$ are cross-intersecting. Every member in $\mathcal{R}^2_{0a}$ contains at least $p+1$ elements in $C$, and every member in $\mathcal{R}^2_{0b}$ contains at least $p$ elements in $C$, so they must intersect in $C$. Therefore $\mathcal{R}_{0a}^*$ and $\mathcal{R}_{0b}^*$ are cross-intersecting.\q
	
Now we consider $c=2p$. Note that $\sum_{x\in X} x\equiv b+4p$ (mod 3). If $b+4p\equiv 1$(mod 3), then $\sum_{y\in Y} y\equiv 2$(mod 3). Let $\alpha\in Y\cap A$ and $\beta \in Y\cap B$. Then $X\cup\{\alpha\}, Y\setminus\{\beta\}\in \mathcal{R}_1\subset \mathcal{F}$ and $X=(X\cup\{\alpha\})\setminus(Y\setminus\{\beta\})\in\mathcal{D}(\mathcal{F})$. We are fine.
	
	For the case $c=2p$ and $b+4p\equiv 2$ (mod 3), applying Lemma \ref{add} by taking $s=2p$, $U_1=A$, $V_1=B$, $U=X\cap A$ and $V=X\cap B$, we have the following claim.
	
	\begin{claim}\label{case1.3add}
		Suppose that $c=2p$ and $b+4p\equiv 2$(mod 3). If $a\le b$, then there exists a vertex $\beta'\in Y\cap B$ such that
		$e(G[X\cap A, (X\cap B)\cup\{\beta'\}])\equiv a$(mod 2). If $a>b$, then there exists a vertex $\beta'\in Y\cap B$ such that
		$e(G[X\cap A, (X\cap B)\cup\{\beta'\}])\equiv b$(mod 2). Furthermore $N((X\cap B)\cup\{\beta'\})\not\subset X\cap A.$
	\end{claim}
	
Let us continue the case $c=2p$ and $b+4p\equiv 2$(mod 3). Take $\beta'$ as in Claim \ref{case1.3add} and let $\mathcal{R}_{0c}^1$ be the family of all such $X\cup\{\beta'\}$. Note that $\mathcal{R}_{0c}^1\subset\mathcal{R}_{0c}^*$. Note that $\sum_{y\in Y} y\equiv 1$(mod 3), take $\alpha'\in Y\cap A$, then $Y\setminus\{\alpha'\}\in \mathcal{R}_1$ and $X=(X\cup\{\beta'\})\setminus(Y\setminus\{\alpha'\})\in \mathcal{D}(\mathcal{F}).$ Clearly $\mathcal{R}_{0c}^1$ is intersecting since each $k$-set in $\mathcal{R}_{0c}^1$ contains $C$.
	
	For the case $c=2p$ and $b+4p\equiv 0$ (mod 3), applying Lemma \ref{delete} by taking $s=2p$, $U_1=A$, $V_1=B$, $U=Y\cap A$ and $V=Y\cap B$, we have the following claim.
	
\begin{claim}\label{case1.3delete}
Suppose that $c=2p$, $b+4p\equiv 0$(mod 3). Let $N(Y\cap A)\not\subset Y\cap B$. If $a\le b-2$, then there exists $\alpha''\in Y\cap A$ such that $e(G[(Y\cap A)\setminus\{\alpha''\}, Y\cap B])\equiv a$(mod 2). If $a>b-2$, then there exists $\alpha''\in Y\cap A$ such that $e(G[(Y\cap A)\setminus\{\alpha''\}, Y\cap B])\equiv b$(mod 2).
\end{claim}
	
Let us continue the case $c=2p$ and $b+4p\equiv 0$(mod 3). Note that $X\cup\{\beta''\}\in\mathcal{R}_1$ for $\beta''\in Y\cap B$. If $N(Y\cap A)\not\subset Y\cap B$, then take $\alpha''$ as in Claim \ref{case1.3delete}. If $N(Y\cap A)\subset Y\cap B$, then take any fixed $\alpha''$ in $Y\cap A$. Let $\mathcal{R}_{0c}^2$ be the family of all such $Y\setminus\{\alpha''\}$. Note that $\mathcal{R}_{0c}^2\subset\mathcal{R}_{0c}^*$
and $X=(X\cup\{\beta''\})\setminus(Y\setminus\{\alpha''\})\in \mathcal{D}(\mathcal{F})$.
Clearly $\mathcal{R}_{0c}^2$ is intersecting since for any $R_0^1$ and $R_0^2$ in $\mathcal{R}_{0c}^2$ we have $|R_0^1|=|R_0^2|=3p$ and $R_0^1\cup R_0^2\subset A\cup B$, then $|R_0^1\cap R_0^2|\ge 2p$.
	
\begin{claim}\label{case1-r0cintersecting}
Let $\mathcal{R}_{0c}^*=\mathcal{R}_{0c}^1\cup\mathcal{R}_{0c}^2$, then $\mathcal{R}_{0c}^*$ is intersecting. \end{claim}
\noindent {\em Proof of Claim \ref{case1-r0cintersecting}.} Since we have explained that $\mathcal{R}_{0c}^1$ is intersecting and $\mathcal{R}_{0c}^2$ is intersecting, it is sufficient to show that $\mathcal{R}_{0c}^1$ and $\mathcal{R}_{0c}^2$ are cross-intersecting.
	Suppose that there exist $R_1\in \mathcal{R}_{0c}^1$ and $R_2\in \mathcal{R}_{0c}^2$ such that $R_1\cap R_2=\emptyset$. If $N(R_2 \cap A)\subset R_2\cap B$, then $N(R_1 \cap B)\subset R_1\cap A$, a contradiction to Claim \ref{case1.3add}. Let $a'=|R_1\cap A|$ and $b'=|R_1\cap B|-1$, then $|R_2\cap A|=2p-|R_1\cap A|=2p-a'$ and $|R_2\cap B|=2p-|R_1\cap B|=2p-b'-1$.
	By Claim \ref{case1.3add}, $e(G[R_1\cap A, R_1\cap B])\equiv |R_1\cap A|=a'$(mod 2) if $|R_1\cap A|\le |R_1\cap B|-1$, i.e. $e(G[R_1\cap A, R_1\cap B])\equiv a'$(mod 2) if $a'\le b'$, and $e(G[R_1\cap A, R_1\cap B])\equiv |R_1\cap B|-1=b'$(mod 2) if $|R_1\cap A|>|R_1\cap B|-1$, i.e. $e(G[R_1\cap A, R_1\cap B])\equiv b'$(mod 2) if $a'>b'$. By Claim \ref{case1.3delete}, $e(G[R_2\cap A, R_2\cap B])\equiv 2p-|R_2\cap A|-1=a'-1$(mod 2) if $2p-|R_2\cap A|-1\le 2p-|R_2\cap B|-2$, i.e. $e(G[R_2\cap A, R_2\cap B])\equiv a'-1$(mod 2) if $a'\le b'$, and $e(G[R_2\cap A, R_2\cap B])\equiv 2p-|R_2\cap B|=b'+1$(mod 2) if $2p-|R_2\cap A|-1>2p-|R_2\cap B|-2$, i.e. $e(G[R_2\cap A, R_2\cap B])\equiv b'+1$(mod 2) if $a'>b'$. So $e(G[R_1\cap A, R_1\cap B])+e(G[R_2\cap A, R_2\cap B])\equiv 1$(mod 2). Since $p$ is even, $|R_1\cap A|+|R_1\cap B|=p$ is even. By Lemma \ref{sum}, $e(G[R_1\cap A, R_1\cap B])+e(G[R_2\cap A, R_2\cap B])$ is even, a contradiction.\q

\begin{claim}\label{case1-allintersecting}
$\mathcal{R}_{0a}^*\cup\mathcal{R}_{0b}^*\cup\mathcal{R}_{0c}^*$ is intersecting.
\end{claim}
\noindent {\em Proof of Claim \ref{case1-allintersecting}.} It is sufficient to show that $\mathcal{R}_{0a}^*\cup\mathcal{R}_{0b}^*$ and $\mathcal{R}_{0c}^*$ are cross-intersecting. Without loss of generality, we prove that $\mathcal{R}_{0a}^*$ and $\mathcal{R}_{0c}^*$ are cross-intersecting here. Recall that $\mathcal{R}^*_{0a}=\mathcal{R}^1_{0a}\cup\mathcal{R}^2_{0a}$, $\mathcal{R}^*_{0c}=\mathcal{R}^1_{0c}\cup\mathcal{R}^2_{0c}$, and every member in $\mathcal{R}^1_{0a}$ contains all elements in $A$ and at least one element in $C$, every member in $\mathcal{R}^2_{0a}$ consists of some elements  but not all elements in $B$ and some elements in $C$, every member in $\mathcal{R}^1_{0c}$ contains all elements in $C$ and at least one element in $B$, every member in $\mathcal{R}^2_{0c}$ consists of some elements but not all elements in $A$ and some elements in $B$. Clearly, $\mathcal{R}^1_{0a}$ and $\mathcal{R}^1_{0c}$ are cross-intersecting, $\mathcal{R}^1_{0a}$ and $\mathcal{R}^2_{0c}$ are cross-intersecting and $\mathcal{R}^2_{0a}$ and $\mathcal{R}^1_{0c}$ are cross-intersecting. Every member in $\mathcal{R}^2_{0a}$ contains at least $p$ elements in $B$, and every member in $\mathcal{R}^2_{0c}$ contains at least $p+1$ elements in $B$, so they must intersect in $B$. Therefore, we may infer that $\mathcal{R}_{0a}^*$ and $\mathcal{R}_{0c}^*$ are cross-intersecting. Similarly, $\mathcal{R}_{0b}^*$ and $\mathcal{R}_{0c}^*$ are cross-intersecting. Hence, $\mathcal{R}_{0a}^*\cup\mathcal{R}_{0b}^*\cup\mathcal{R}_{0c}^*$ is intersecting.\q
	
	In summary, we have constructed an intersecting family $\mathcal{F}\subset{\mathcal{N}\choose 3p}$ such that ${\mathcal{N}\choose 3p-1}\subset\mathcal{D}(\mathcal{F})$  for $|\mathcal{N}|=6p$.
	
	Next, we will show that $\cup_{j=1}^{3p-2}{\mathcal{N}\choose j}\subset\mathcal{D}(\mathcal{F})$. Recall that $\mathcal{R}_1\subset \mathcal{F}$, so $\mathcal{D}(\mathcal{R}_1)\subset \mathcal{D}(\mathcal{F})$. In fact we have the following claim.

\begin{claim}\label{case1-r1diff}
$\cup_{j=1}^{k-2}{\mathcal{N}\choose j}\subset\mathcal{D}(\mathcal{R}_1)$.
\end{claim}
\noindent {\em Proof of Claim \ref{case1-r1diff}.} It is sufficient to show that for any $X\subset \mathcal{N}$ we have $X\in \mathcal{D}(\mathcal{R}_1)$, where $1\le |X|\le k-2$. Recall that $k=3p$.
	
	We consider $|X|=1$ first. We may assume that $X\subset A$ (the cases $X\subset B$ and $X\subset C$ are similar). Without loss of generality, let $X=\{0\}$. Take $B'=\{1, 4, \dots, \frac{3k}{2}-2\}$ and $C'=\{2, 5, \dots, \frac{3k}{2}-4\}$. Since $X\cup B'\cup C'\in \mathcal{R}_1$ and $\{3\}\cup B'\cup C'\in \mathcal{R}_1$, $X\subset \mathcal{D}(\mathcal{R}_1)$.
	
	We consider $2\le |X|=j\le k-2$ next. We will use induction on $j$. We will show for $|X|=k-2$ first. Let $Y=\mathcal{N}\setminus X$. Let $\sum_{x\in X} x\equiv i$ (mod 3), then $\sum_{y\in Y} y\equiv 2i$ (mod 3). For convenience, denote $N_0=A$, $N_1=B$ and $N_2=C$. If $|Y\cap N_{i+1}|\ge 2$ and $|Y\cap N_{i+2}|\ge 2$, then let $b_1, b_2\in Y\cap N_{i+1}$ and $c_1, c_2\in Y\cap N_{i+2}$ (here, the subscripts $i+1$ and $i+2$ in $N_{i+1}$ and $N_{i+2}$ are in the sense of mod 3). So $X\cup\{c_1, c_2\}\in \mathcal{R}_1$ and $Y\setminus\{b_1, b_2\}\in \mathcal{R}_1$ and $X=(X\cup\{c_1, c_2\})\setminus(Y\setminus\{b_1, b_2\})$, i.e. $X\in \mathcal{D}(\mathcal{R}_1)$. Therefore we may assume that $|Y\cap N_{i+1}|\le 1$ or $|Y\cap N_{i+2}|\le 1$. Without loss of generality, we assume that $|Y\cap N_{i+1}|\le 1$, then $|Y\cap N_{i+2}|\ge 3$ and $|Y\cap N_i|\ge 3$ by noting that $|Y|=k+2$. Let $a_1\in Y\cap N_i$ and $c_1, c_2, c_3\in Y\cap N_{i+2}$. So $X\cup\{c_1, c_2\}\in \mathcal{R}_1$ and $Y\setminus\{a_1, c_3\}\in \mathcal{R}_1$ and $X=(X\cup\{c_1, c_2\})\setminus(Y\setminus\{a_1, c_3\})$, i.e. $X\in \mathcal{D}(\mathcal{R}_1)$.
	
	Assume that ${\mathcal{N}\choose j+1}\subset\mathcal{D}(\mathcal{R}_1)$ where $2\le j\le k-3$, it is sufficient to show that ${\mathcal{N}\choose j}\subset\mathcal{D}(\mathcal{R}_1)$. Let $|X|=j$ and $Y=\mathcal{N}\setminus X$. Note that one of $|Y\cap A|$, $|Y\cap B|$ and $|Y\cap C|$ is at least 3. Without loss of generality, we may assume that $|Y\cap A|\ge 3$. Let $a_1\in Y\cap A$ and $X'=X\cup\{a_1\}$. Then by the induction hypothesis, there exists $I\subset Y\setminus\{a_1\}$ and $Y'\subset Y\setminus (I\cup \{a_1\})$ such that $X'\cup I\in \mathcal{R}_1$, $Y'\cup I\in \mathcal{R}_1$ and $X\cup\{a_1\}=(X'\cup I)\setminus(Y'\cup I)$.
	
	We claim that $Y'\cap A=\emptyset$. Otherwise if there exists $a_2\in Y'\cap A$, then $X\cup I\cup\{a_2\}\in \mathcal{R}_1$ and $X=(X\cup I\cup\{a_2\})\setminus(Y'\cup I)$, hence $X\in \mathcal{D}(\mathcal{R}_1)$.
	
	We claim that $A\setminus (X'\cup I)=\emptyset$. Otherwise if there exists $a_3\in A\setminus (X'\cup I)$, then we claim that $Y'\cap B=\emptyset$ or $Y'\cap C=\emptyset$. Otherwise if $b_1\in Y'\cap B$ and $c_1\in Y'\cap C$. Then $Y'\cup I\cup\{a_1, a_3\}\setminus\{b_1, c_1\}\in \mathcal{R}_1$ and $X=(X'\cup I)\setminus(Y'\cup I\cup\{a_1, a_3\}\setminus\{b_1, c_1\})$. Hence $X\in \mathcal{D}(\mathcal{R}_1)$. Without loss of generality, we may assume that $Y'\cap C=\emptyset$. So $Y'\subset B$ with $|Y'|=j+1\ge 3$. If $|A\setminus (X'\cup I)|\ge 2$, then let $a_3, a_4\in A\setminus (X'\cup I)$. Assume that $b_1, b_2, b_3\in Y'\cap B$. Then $Y'\cup I\cup \{a_1, a_3, a_4\}\setminus\{b_1, b_2, b_3\}\in \mathcal{R}_1$ and $X=(X\cup I\cup\{a_1\})\setminus(Y'\cup I\cup \{a_1, a_3, a_4\}\setminus\{b_1, b_2, b_3\})$, hence $X\in \mathcal{D}(\mathcal{R}_1)$. So $|A\setminus (X'\cup I)|=|\{a_3\}|=1$ and $A\setminus\{a_3\}\subset X\cup I\cup \{a_1\}$. Since $|A\setminus\{a_3\}|=2p-1$ and $|X\cup I\cup\{a_1\}|=k=3p$, $|C\setminus (X'\cup I)|\ge 1$. Let $c_1\in C\setminus (X'\cup I)$. Then $I\cup Y'\cup\{a_1, c_1\}\setminus\{b_1, b_2\}\in \mathcal{R}_1$ and $X=(X\cup I\cup\{a_1\})\setminus(I\cup Y'\cup\{a_1, c_1\}\setminus\{b_1, b_2\})$, hence $X\in \mathcal{D}(\mathcal{R}_1)$. So $A\setminus (X'\cup I)=\emptyset$. Therefore $A\subset X'\cup I$. Since $|A\setminus X|\ge 3$, $|I\cap A|\ge 2$. Assume that $\theta_1, \theta_2\in I\cap A$.
	
	We claim that $|Y'\cap B|=0$ or $|Y'\cap C|=0$. Otherwise, there exists $b_1\in Y'\cap B$ and $c_1\in Y'\cap C$, let $I'=I\cup\{b_1, c_1\}\setminus\{\theta_1, \theta_2\}$ and $Y''=Y'\cup\{\theta_1, \theta_2\}\setminus\{b_1, c_1\}$. So $X'\cup I'\in \mathcal{R}_1$, $I'\cup Y''\in\mathcal{R}_1$, and $|Y''\cap A|\ge 1$, a contradiction. Without loss of generality, we may assume that $Y'\cap C=\emptyset$, then $Y'\subset B$ and $|Y'|=j+1\ge 3$. Let $b_1, b_2\in Y'\cap B$. Since $X\cup I\cup\{a_1\}=k=3p$, $A\subset X\cup I\cup\{a_1\}$ and $Y'\cap C=\emptyset$, there exists $c\in C\setminus(X\cup I\cup\{a_1\}\cup Y')$. Let $Y''=Y\setminus\{b_1, b_2\}\cup\{a_1, c\}$. Then $Y''\cup I\in \mathcal{R}_1$ and $X=(X\cup I\cup\{a_1\})\setminus(Y''\cup I)$, a contradiction.\q

{\bf Case 2. $k\equiv 1$(mod 3).}

Let $k=3p+1$ be even and $\mathcal{N}=\{0, 1, 2, \dots, 6p+1\}$,
$$A=\{0, 3, 6, \dots, 6p\},$$ $$B=\{1, 4, 7, \dots, 6p+1\}$$ and $$C=\{2, 5, 8, \dots, 6p-1\}.$$
We will construct an intersecting family $\mathcal{F}$ such that $\cup_{j=1}^{k-1}{\mathcal{N}\choose j}\subset \mathcal{D}(\mathcal{F})$. We first focus on constructing an intersecting family satisfying the requirement ${\mathcal{N}\choose k-1}\subset \mathcal{D}(\mathcal{F})$. Then we show that our construction satisfies $\cup_{j=1}^{k-2}{\mathcal{N}\choose j}\subset \mathcal{D}(\mathcal{F})$ as well. Let $X\in {\mathcal{N}\choose k-1}$ and $|X\cap A|=a$, $|X\cap B|=b$ and $|X\cap C|=c$, then $a+b+c=k-1=3p$, we will construct intersecting $\mathcal{F}\subseteq {\mathcal{N}\choose 3p+1}$ such that $X\in \mathcal{D}(\mathcal{F})$ for all $X\in{\mathcal{N}\choose k-1}$. Let $Y=\mathcal{N}\setminus X$. Note that $\sum_{t\in \mathcal{N}} t\equiv 1$(mod 3) and $\sum_{x\in X} x\equiv b+2c$(mod 3).
Recall that
$$\mathcal{R}_i=\{\{x_1, x_2, \dots, x_k\}\in {\mathcal{N}\choose 3p}: \sum_{j=1}^k x_j\equiv i(mod \ 3)\}.$$	
By Lemma \ref{claimR1}, $\mathcal{R}_0$ is intersecting, and $\mathcal{R}_0$ and $\mathcal{R}_2$ are cross-intersecting.
	
	We first put all members in $\mathcal{R}_0$ to $\mathcal{F}$, i.e. $\mathcal{R}_0\subset \mathcal{F}$. By Lemma \ref{normal}, $X\in \mathcal{D}(\mathcal{R}_0)$ if $a, b<2p+1$ and $c<2p$. For the situation $a=2p+1$ or $b=2p+1$ or $c=2p$, we will add an intersecting subfamily $\mathcal{R}^*_{2a}\cup\mathcal{R}^*_{2b}\cup\mathcal{R}^*_{2c}=\mathcal{R}^*_2\subseteq\mathcal{R}_2$ into $\mathcal {F}$ such that $\mathcal{F}=\mathcal{R}^*_2\cup \mathcal{R}_0$. By Lemma \ref{claimR1}, $\mathcal{F}$ is intersecting provided $\mathcal{R}^*_2$ is intersecting.
	
	We construct an auxiliary tripartite graph $G$ with vertex set $A\cup B\cup C$ and edge set $E(G)=\cup_{0\le \alpha\le 2p-1}\{\{3\alpha,3\alpha+1\}, \{3\alpha+1, 3\alpha+2\}, \{3\alpha, 3\alpha+2\}\}\cup\{6p, 6p+1\}$. Note that $|A|=|B|\equiv 1$(mod 2) and $|C|\equiv 0$(mod 2) and $p\equiv 1$(mod 2).
	
We consider $a=2p+1$ first. Note that $\sum_{x\in X} x\equiv b+2c$ (mod 3). If $b+2c\equiv 2$(mod 3), then $\sum_{y\in Y} y\equiv 2$(mod 3). Let $\beta\in Y\cap B$ and $\gamma\in Y\cap C$. Then $X\cup\{\beta\}, Y\setminus\{\gamma\}\in \mathcal{R}_0\subset \mathcal{F}$ and $X=(X\cup\{\beta\})\setminus(Y\setminus\{\gamma\})\in\mathcal{D}(\mathcal{F})$. In this case, we are fine.
	
	For the case $a=2p+1$ and $b+2c\equiv 0$ (mod 3), applying Lemma \ref{add} by taking $s=2p$, $U_2=B$, $V_2=C$, $U=X\cap B$ and $V=X\cap C$, we have the following claim.
	
\begin{claim}\label{case2.1add}
Suppose that $a=2p+1$ and $b+2c\equiv 0$ (mod 3). If $b-|X\cap\{6p+1\}|\le c$, then there exists a vertex $\gamma'\in Y\cap C$ such that $e(G[X\cap B, (X\cap C)\cup\{\gamma'\}])\equiv (b-|X\cap\{6p+1\}|)$(mod 2). If $b-|X\cap\{6p+1\}|>c$, then there exists a vertex $\gamma'\in Y\cap C$ such that $e(G[X\cap B, (X\cap C)\cup\{\gamma'\}])\equiv c$(mod 2). Furthermore $N((X\cap C)\cup\{\gamma'\})\not\subset X\cap B.$
\end{claim}
	
Let us continue the case $a=2p+1$ and $b+2c\equiv 0$ (mod 3). Take $\gamma'$ as in Claim \ref{case2.1add}, and let $\mathcal{R}_{2a}^1$ be the family of all such $X\cup\{\gamma'\}$. Note that $\in\mathcal{R}_{2a}^1\subset\mathcal{R}_{2a}^*$. Since $\sum_{y\in Y} y\equiv 1$(mod 3), take $\beta'\in Y\cap B$, then $Y\setminus\{\beta'\}\in \mathcal{R}_0$ and $X=(X\cup\{\gamma'\})\setminus(Y\setminus\{\beta'\})\in \mathcal{D}(\mathcal{F}).$ Clearly $\mathcal{R}_{2a}^1$ is intersecting since each $k$-set in $\mathcal{R}_{2a}^1$ contains $A$.
	
	For the case $a=2p+1$ and $b+2c\equiv 1$ (mod 3), applying Lemma \ref{delete} by taking $s=2p$, $U_2=B$, $V_2=C$, $U=Y\cap B$ and $V=Y\cap C$, we have the following claim.
	
\begin{claim}\label{case2.1delete}
Suppose that $a=2p+1$ and $b+2c\equiv 1$ (mod 3). Let $Y=\mathcal{N}\setminus X$ and $N(Y\cap B)\not\subset Y\cap C$. If $b-|X\cap\{6p+1\}|\le c-2$, then there exists $\beta''\in (Y\cap B)\setminus\{6p+1\}$ such that $e(G[(Y\cap B)\setminus\{\beta''\}, Y\cap C])\equiv (b+1)$(mod 2). If $b-|X\cap\{6p+1\}|>c-2$, then there exists $\beta''\in (Y\cap B)\setminus\{6p+1\}$ such that $e(G[(Y\cap B)\setminus\{\beta''\}, Y\cap C])\equiv (c-1-|X\cap\{6p+1\}|)$(mod 2).
	\end{claim}
	
Let us continue the case $a=2p+1$ and $b+2c\equiv 1$ (mod 3). Note that $X\cup\{\gamma''\}\in\mathcal{R}_0$ for $\gamma''\in Y\cap C$. If $N(Y\cap B)\not\subset Y\cap C$, then take $\beta''$ as in Claim \ref{case2.1delete}. If $N(Y\cap B)\subset Y\cap C$, then take any fixed $\beta''\in Y\cap B\setminus\{6p+1\}$. Let $\mathcal{R}_{2a}^2$ be the family of all such $Y\setminus\{\beta''\}$. Note that $\mathcal{R}_{2a}^2\subset\mathcal{R}_{2a}^*$ and $X=(X\cup\{\gamma''\})\setminus(Y\setminus\{\beta''\})\in \mathcal{D}(\mathcal{F})$. Clearly $\mathcal{R}_{2a}^2$ is intersecting since for any $R_2^1$ and $R_2^2$ in $\mathcal{R}_{2a}^2$ we have $|R_2^1|=|R_2^2|=3p+1$ and $R_2^1\cup R_2^2\subset B\cup C$, then $|R_2^1\cap R_2^2|\ge 2p+1$.
	
\begin{claim}\label{case2-r2aintersecting}
Let $\mathcal{R}_{2a}^*=\mathcal{R}_{2a}^1\cup\mathcal{R}_{2a}^2$, then $\mathcal{R}_{2a}^*$ is intersecting.
\end{claim}
\noindent {\em Proof of Claim \ref{case2-r2aintersecting}.} Since we have explained that $\mathcal{R}_{2a}^1$ is intersecting and $\mathcal{R}_{2a}^2$ is intersecting, it is sufficient to show that $\mathcal{R}_{2a}^1$ and $\mathcal{R}_{2a}^2$ are cross-intersecting. Suppose that there exist $R_1\in \mathcal{R}_{2a}^1$ and $R_2\in \mathcal{R}_{2a}^2$ such that $R_1\cap R_2=\emptyset.$ If $N(R_2\cap B)\subset R_2\cap C$, then $N(R_1\cap C)\subset R_1\cap B$, a contradiction to Claim \ref{case2.1add}. So $N(R_2\cap B)\not\subset R_2\cap C$. Let $|R_1\cap B|=b'$ and $|R_1\cap C|-1=c'$, then $b'+2c'\equiv 0$(mod 3). By Claim \ref{case2.1add}, $e(G[R_1\cap B, R_1\cap C])\equiv (b'-|R_1\cap\{6p+1\}|)$(mod 2) if $b'-|R_1\cap\{6p+1\}|\le c'$ and $e(G[R_1\cap B, R_1\cap C])\equiv c'$(mod 2) if $b'-|R_1\cap\{6p+1\}|>c'$. By Claim \ref{case2.1delete}, $e(G[R_2\cap B, R_2\cap C])\equiv b'$(mod 2) if $b'-|R_1\cap\{6p+1\}|\le c'$ and $e(G[R_2\cap B, R_2\cap C])\equiv (c'-|R_1\cap\{6p+1\}|)$(mod 2) if $b'-|R_1\cap\{6p+1\}|>c'$. So $e(G[R_1\cap B, R_1\cap C])+e(G[R_2\cap B, R_2\cap C])\equiv |R_1\cap\{6p+1\}|$(mod 2). Since $p$ is odd, $|R_1\cap B|+|R_1\cap C|=p$ is odd. By Lemma \ref{sum}, $e(G[R_1\cap B, R_1\cap C])+e(G[R_2\cap B, R_2\cap C])\equiv p-|R_1\cap\{6p+1\}|\equiv (1+|R_1\cap\{6p+1\}|)$(mod 2), a contradiction.\q
	
Next, we consider $b=2p+1$. Note that $\sum_{x\in X} x\equiv 2p+1+2c$ (mod 3). If $2p+1+2c\equiv 1$(mod 3), then $\sum_{y\in Y} y\equiv 0$(mod 3). Let $\alpha\in Y\cap A$ and $\gamma \in Y\cap C$. Then $X\cup\{\gamma\}, Y\setminus\{\alpha\}\in \mathcal{R}_0\subset \mathcal{F}$ and $X=(X\cup\{\gamma\})\setminus(Y\setminus\{\alpha\})\in\mathcal{D}(\mathcal{F})$. In this case, we are fine.
	
	For the case $b=2p+1$ and $2p+1+2c\equiv 2$ (mod 3), applying Lemma \ref{add} by taking $s=2p$, $U_3=C$, $V_3=A$, $U=X\cap C$ and $V=X\cap A$, we have the following claim.
	
	\begin{claim}\label{case2.2add}
		Suppose that $b=2p+1$ and $2p+1+2c\equiv 2$(mod 3). If $c\le a-|X\cap\{6p\}|$, then there exists a vertex $\alpha'\in (Y\cap A)\setminus\{6p\}$ such that
		$e(G[(X\cap A)\cup\{\alpha'\}, X\cap C])\equiv c$(mod 2). If $c>a-|X\cap\{6p\}|$, then there exists a vertex $\alpha'\in (Y\cap A)\setminus\{6p\}$ such that $e(G[(X\cap A)\cup\{\alpha'\}, X\cap C])\equiv (a-|X\cap\{6p\}|)$(mod 2). Furthermore $N((X\cap A)\cup\{\alpha'\})\not\subset X\cap C.$
	\end{claim}
	
Let us continue the case $b=2p+1$ and $2p+1+2c\equiv 2$(mod 3). Take $\alpha'$ as in Claim \ref{case2.2add} and let $\mathcal{R}_{2b}^1$ be the family of all such $X\cup\{\alpha'\}$. Note that $\mathcal{R}_{2b}^1\subset\mathcal{R}_{2b}^*$. Since $\sum_{y\in Y} y\equiv 2$(mod 3), take $\gamma'\in Y\cap C$, then $Y\setminus\{\gamma'\}\in \mathcal{R}_0$ and $X=(X\cup\{\alpha'\})\setminus(Y\setminus\{\gamma'\})\in \mathcal{D}(\mathcal{F}).$ Clearly $\mathcal{R}_{2b}^1$ is intersecting since each $k$-set in $\mathcal{R}_{2b}^1$ contains $B$.
	
	For the case $b=2p+1$ and $2p+1+2c\equiv 0$ (mod 3), applying Lemma \ref{delete} by taking $s=2p$, $U_3=C$, $V_3=A$, $U=Y\cap C$ and $V=Y\cap A$, we have the following claim.
	
\begin{claim}\label{case2.2delete}
Suppose that $b=2p+1$ and $2p+1+2c\equiv 0$(mod 3). Let $N(Y\cap C)\not\subset Y\cap A$. If $c\le a-|X\cap\{6p\}|-2$, then there exists $\gamma''\in Y\cap C$ such that $e(G[Y\cap A, (Y\cap C)\setminus\{\gamma''\}])\equiv (c+1-|X\cap\{6p\}|)$(mod 2). If $c>a-|X\cap\{6p\}|-2$, then there exists $\gamma''\in Y\cap C$ such that $e(G[Y\cap A, (Y\cap C)\setminus\{\gamma''\}]\equiv (a-1)$(mod 2).
\end{claim}
	
Let us continue the case $b=2p+1$ and $2p+1+2c\equiv 0$(mod 3). Note that $X\cup\{\alpha''\}\in\mathcal{R}_0$ for $\alpha''\in Y\cap A$. If $N(Y\cap C)\not\subset Y\cap A$, then take $\gamma''$ as in Claim \ref{case2.2delete}. If $N(Y\cap C)\subset Y\cap A$, then take any fixed $\gamma''\in Y\cap C$. Let $\mathcal{R}_{2b}^2$ be the family of all such $Y\setminus\{\gamma''\}$. Note that $\mathcal{R}_{2b}^2\subset\mathcal{R}_{2b}^*$ and $X=(X\cup\{\alpha''\})\setminus(Y\setminus\{\gamma''\})\in \mathcal{D}(\mathcal{F})$. Clearly $\mathcal{R}_{2b}^2$ is intersecting since for any $R_2^1$ and $R_2^2$ in $\mathcal{R}_{2b}^2$ we have $|R_2^1|=|R_2^2|=3p+1$ and $R_2^1\cup R_2^2\subset A\cup C$, then $|R_2^1\cap R_2^2|\ge 2p+1$.
	
\begin{claim}\label{case2-r2bintersecting}
Let $\mathcal{R}_{2b}^*=\mathcal{R}_{2b}^1\cup\mathcal{R}_{2b}^2$, then $\mathcal{R}_{2b}^*$ is intersecting.
\end{claim}
\noindent {\em Proof of Claim \ref{case2-r2bintersecting}.} Since we have explained that $\mathcal{R}_{2b}^1$ is intersecting and $\mathcal{R}_{2b}^2$ is intersecting, it is sufficient to show that $\mathcal{R}_{2b}^1$ and $\mathcal{R}_{2b}^2$ are cross-intersecting. Suppose that there exist $R_1\in \mathcal{R}_{2b}^1$ and $R_2\in \mathcal{R}_{2b}^2$ such that $R_1\cap R_2=\emptyset$. If $N(R_2\cap C)\subset R_2\cap A$, then $N(R_1\cap A)\subset R_1\cap C$, contradicting to Claim \ref{case2.2add}. So $N(R_2\cap C)\not\subset R_2\cap A$. Let $c'=|R_1\cap C|$ and $a'=|R_1\cap A|-1$. By Claim \ref{case2.2add}, $e(G[R_1\cap A, R_1\cap C])\equiv c'$ (mod 2) if $c'\le a'-|R_1\cap\{6p\}|$ and $e(G[R_1\cap A, R_1\cap C])\equiv (a'-|R_1\cap\{6p\}|)$ (mod 2) if $c'>a'-|R_1\cap\{6p\}|$. By Claim \ref{case2.2delete}, $e(G[R_2\cap A, R_2\cap C])\equiv (c'-|R_1\cap\{6p\}|)$ (mod 2) if $c'\le a'-|R_1\cap\{6p\}|$ and $e(G[R_2\cap B, R_2\cap C])\equiv a'$(mod 2) if $c'>a'-|R_1\cap\{6p\}|$. So $e(G[R_1\cap A, R_1\cap C])+e(G[R_2\cap A, R_2\cap C])\equiv (|R_1\cap\{6p\}|)$(mod 2). Since $p$ is odd, $|R_1\cap A|+|R_1\cap C|=p$ is odd. By Lemma \ref{sum}, $e(G[R_1\cap A, R_1\cap C])+e(G[R_2\cap A, R_2\cap C])\equiv p-|R_1\cap\{6p\}|\equiv (1+|R_1\cap\{6p\}|)$ (mod 2), a contradiction.\q

\begin{claim}\label{case2-r2ar2bintersecting}
$\mathcal{R}_{2a}^*$ and $\mathcal{R}_{2b}^*$ are cross-intersecting.
\end{claim}
\noindent {\em Proof of Claim \ref{case2-r2ar2bintersecting}.} Recall that $\mathcal{R}^*_{2a}=\mathcal{R}^1_{2a}\cup\mathcal{R}^2_{2a}$, $\mathcal{R}^*_{2b}=\mathcal{R}^1_{2b}\cup\mathcal{R}^2_{2b}$, and every member in $\mathcal{R}^1_{2a}$ contains all elements in $A$ and at least one element in $C$, every member in $\mathcal{R}^2_{2a}$ consists of some elements  but not all elements in $B$ and some elements in $C$, every member in $\mathcal{R}^1_{2b}$ contains all elements in $B$ and at least one element in $A$, every member in $\mathcal{R}^2_{2b}$ consists of some elements but not all elements in $C$ and some elements in $A$. Clearly, $\mathcal{R}^1_{2a}$ and $\mathcal{R}^1_{2b}$ are cross-intersecting, $\mathcal{R}^1_{2a}$ and $\mathcal{R}^2_{2b}$ are cross-intersecting and $\mathcal{R}^2_{2a}$ and $\mathcal{R}^1_{2b}$ are cross-intersecting. Every member in $\mathcal{R}^2_{2a}$ contains at least $p+1$ elements in $C$, and every member in $\mathcal{R}^2_{2b}$ contains at least $p$ elements in $C$, so they must intersect in $C$. Therefore $\mathcal{R}_{2a}^*$ and $\mathcal{R}_{2b}^*$ are cross-intersecting.\q

For $c=2p$, $\sum_{x\in X} x\equiv b+4p$ (mod 3). If $b+4p\equiv 0$(mod 3), then $\sum_{y\in Y} y\equiv 1$(mod 3). Let $\alpha\in Y\cap A$ and $\beta \in Y\cap B$. Then $X\cup\{\alpha\}, Y\setminus\{\beta\}\in \mathcal{R}_0\subset \mathcal{F}$ and $X=(X\cup\{\alpha\})\setminus(Y\setminus\{\beta\})\in\mathcal{D}(\mathcal{F})$. In this case, we are fine.
	
	For the case $c=2p$ and $b+4p\equiv 1$ (mod 3), applying Lemma \ref{add} by taking $s=2p+1$, $U_1=A$, $V_1=B$, $U=X\cap A$ and $V=X\cap B$, we have the following claim.
	
	\begin{claim}\label{case2.3add}
		Suppose that $c=2p$ and $b+4p\equiv 1$(mod 3). If $a\le b$, then there exists a vertex $\beta'\in Y\cap B$ such that
		$e(G[X\cap A, (X\cap B)\cup\{\beta'\}])\equiv a$(mod 2). If $a>b$, then there exists a vertex $\beta'\in Y\cap B$ such that
		$e(G[X\cap A, (X\cap B)\cup\{\beta'\}])\equiv b$(mod 2). Furthermore $N((X\cap B)\cup\{\beta'\})\not\subset X\cap A.$
	\end{claim}
	
Let us continue the case $c=2p$ and $b+4p\equiv 1$(mod 3). Take $\beta'$ as in Claim \ref{case2.3add} and let $\mathcal{R}_{2c}^1$ be the family of all such $X\cup\{\beta'\}$. Note that $\mathcal{R}_{2c}^1\subset\mathcal{R}_{2c}^*$. Since $\sum_{y\in Y} y\equiv 0$(mod 3), take $\alpha'\in Y\cap A$, then $Y\setminus\{\alpha'\}\in \mathcal{R}_0$ and $X=(X\cup\{\beta'\})\setminus(Y\setminus\{\alpha'\})\in \mathcal{D}(\mathcal{F}).$ Clearly $\mathcal{R}_{2c}^1$ is intersecting since each $k$-set in $\mathcal{R}_{2c}^1$ contains $C$.
	
	For the case $c=2p$ and $b+4p\equiv 2$ (mod 3), applying Lemma \ref{delete} by taking $s=2p+1$, $U_1=A$, $V_1=B$, $U=Y\cap A$ and $V=Y\cap B$, we have the following claim.
	
\begin{claim}\label{case2.3delete}
Suppose that $c=2p$ and $b+4p\equiv 2$(mod 3). Let $Y=\mathcal{N}\setminus X$ and $N(Y\cap A)\not\subset Y\cap B$. If $a\le b-2$, then there exists $\alpha''\in Y\cap A$ such that $e(G[(Y\cap A)\setminus\{\alpha''\}, Y\cap B])\equiv (a+1)$(mod 2). If $a>b-2$, then there exists $\alpha''\in Y\cap A$ such that $e(G[(Y\cap A)\setminus\{\alpha''\}, Y\cap B])\equiv (b-1)$(mod 2).
\end{claim}
	
Let us continue the case $c=2p$ and $b+4p\equiv 2$(mod 3). Note that $X\cup\{\beta''\}\in\mathcal{R}_0$ for $\beta''\in Y\cap B$. If $N(Y\cap A)\not\subset Y\cap B$, then take $\alpha''$ as in Claim \ref{case2.3delete}. If $N(Y\cap A)\subset Y\cap B$, then take any fixed $\alpha''\in Y\cap A$. Let $\mathcal{R}_{2c}^2$ be the family of all such $Y\setminus\{\alpha''\}$. Note that $\mathcal{R}_{2c}^2\subset\mathcal{R}_{2c}^*$ and $X=(X\cup\{\beta''\})\setminus(Y\setminus\{\alpha''\})\in \mathcal{D}(\mathcal{F})$. Clearly $\mathcal{R}_{2c}^2$ is intersecting since for any $R_2^1$ and $R_2^2$ in $\mathcal{R}_{2c}^2$ we have $|R_2^1|=|R_2^2|=3p+1$ and $R_2^1\cup R_2^2\subset A\cup B$, then $|R_2^1\cap R_2^2|\ge 2p$.
	
\begin{claim}\label{case2-r2cintersecting}
Let $\mathcal{R}_{2c}^*=\mathcal{R}_{2c}^1\cup\mathcal{R}_{2c}^2$, then $\mathcal{R}_{2c}^*$ is intersecting.\end{claim}
\noindent {\em Proof of Claim \ref{case2-r2cintersecting}.} Since we have explained that $\mathcal{R}_{2c}^1$ is intersecting and $\mathcal{R}_{2c}^2$ is intersecting, it is sufficient to show that $\mathcal{R}_{2c}^1$ and $\mathcal{R}_{2c}^2$ are cross-intersecting. Suppose that there exist $R_1\in \mathcal{R}_{2c}^1$ and $R_2\in \mathcal{R}_{2c}^2$ such that $R_1\cap R_2=\emptyset$. If $N(R_2\cap A)\subset R_2\cap B$, then $N(R_1\cap B)\subset R_1\cap A$, contradicting to Claim \ref{case2.3add}. So $N(R_2\cap A)\not\subset R_2\cap B$. Let $a'=|R_1\cap A|$ and $b'=|R_1\cap B|-1$. By Claim \ref{case2.3add}, $e(G[R_1\cap A, R_1\cap B])\equiv a'$(mod 2) if $a'\le b'$ and $e(G[R_1\cap A, R_1\cap B])\equiv b'$(mod 2) if $a'>b'$. By Claim \ref{case2.3delete}, $e(G[R_2\cap A, R_2\cap B])\equiv a$(mod 2) if $a'\le b'$ and $e(G[R_2\cap A, R_2\cap B])\equiv b'$(mod 2) if $a'>b'$. So $e(G[R_1\cap A, R_1\cap B])+e(G[R_2\cap A, R_2\cap B])\equiv 0$(mod 2). Since $p$ is odd, $|R_1\cap A|+|R_1\cap B|=p+1$ is even. By Lemma \ref{sum}, $e(G[R_1\cap A, R_1\cap B])+e(G[R_2\cap A, R_2\cap B])$ is odd, a contradiction.\q

\begin{claim}\label{case2-allintersecting}
$\mathcal{R}_{2a}^*\cup\mathcal{R}_{2b}^*\cup\mathcal{R}_{2c}^*$ is intersecting.
\end{claim}
\noindent {\em Proof of Claim \ref{case2-allintersecting}.} It is sufficient to show that $\mathcal{R}_{2a}^*\cup\mathcal{R}_{2b}^*$ and $\mathcal{R}_{2c}^*$ are cross-intersecting. Without loss of generality, we prove that $\mathcal{R}_{2a}^*$ and $\mathcal{R}_{2c}^*$ are cross-intersecting here. Recall that $\mathcal{R}^*_{2a}=\mathcal{R}^1_{2a}\cup\mathcal{R}^2_{2a}$, $\mathcal{R}^*_{2c}=\mathcal{R}^1_{2c}\cup\mathcal{R}^2_{2c}$, and every member in $\mathcal{R}^1_{2a}$ contains all elements in $A$ and at least one element in $C$, every member in $\mathcal{R}^2_{2a}$ consists of some elements but not all elements in $B$ and some elements in $C$, every member in $\mathcal{R}^1_{2c}$ contains all elements in $C$ and at least one element in $B$, every member in $\mathcal{R}^2_{2c}$ consists of some elements but not all elements in $A$ and some elements in $B$. Clearly, $\mathcal{R}^1_{2a}$ and $\mathcal{R}^1_{2c}$ are cross-intersecting, $\mathcal{R}^1_{2a}$ and $\mathcal{R}^2_{2c}$ are cross-intersecting and $\mathcal{R}^2_{2a}$ and $\mathcal{R}^1_{2c}$ are cross-intersecting. Every member in $\mathcal{R}^2_{2a}$ contains at least $p+1$ elements in $B$, and every member in $\mathcal{R}^2_{2c}$ contains at least $p+1$ elements in $B$, so they must intersect in $B$. Therefore, we may infer that $\mathcal{R}_{2a}^*$ and $\mathcal{R}_{2c}^*$ are cross-intersecting. Similarly, $\mathcal{R}_{2b}^*$ and $\mathcal{R}_{2c}^*$ are cross-intersecting. Hence, $\mathcal{R}_{2a}^*\cup\mathcal{R}_{2b}^*\cup\mathcal{R}_{2c}^*$ is intersecting.\q
	
	In summary, we have constructed an intersecting family $\mathcal{F}\subset{\mathcal{N}\choose k}$ such that ${[2k]\choose k-1}\subset\mathcal{D}(\mathcal{F})$ if $k\equiv 1$ (mod 3).
	
	Next, we will show that $\cup_{j=1}^{k-2}{[2k]\choose j}\subset\mathcal{D}(\mathcal{F})$. Recall that $\mathcal{R}_0\subset \mathcal{F}$, $\mathcal{D}(\mathcal{R}_0)\subset \mathcal{D}(\mathcal{F})$. We will have the following claim.

\begin{claim}\label{case2-r0diff}
$\cup_{j=1}^{k-2}{\mathcal{N}\choose j}\subset\mathcal{D}(\mathcal{R}_0)$.
\end{claim}
\noindent {\em Proof of Claim \ref{case2-r0diff}.} It is sufficient to show that for any $X\subset \mathcal{N}$ we have $X\subset \mathcal{D}(\mathcal{R}_0)$, where $1\le |X|\le k-2$.
	
	We consider $|X|=1$ first. We may assume that $X\subset A$ (the cases $X\subset B$ and $X\subset C$ are similar). Without loss of generality, let $X=\{0\}$. Take $B'=\{1, 4, \dots, \frac{3k}{2}-4\}$ and $C'=\{2, 5, \dots, \frac{3k}{2}-4\}$. Since $X\cup B'\cup C'\cup\{6\}\in \mathcal{R}_0$ and $\{3, 6\}\cup B'\cup C'\in \mathcal{R}_0$, $X\subset \mathcal{D}(\mathcal{R}_0)$.
	
	We consider $2\le |X|=j\le k-2$ next. We will use induction on $j$. We will show for $|X|=k-2$ first. Let $Y=\mathcal{N}\setminus X$. Let $\sum_{x\in X} x\equiv i$ (mod 3), then $\sum_{y\in Y} y\equiv 2i+1$ (mod 3). For convenience, denote $N_0=A$, $N_1=B$ and $N_2=C$. If $|Y\cap N_{i}|\ge 2$ and $|Y\cap N_{i+2}|\ge 2$, then let $a_1, a_2\in Y\cap N_{i}$ and $c_1, c_2\in Y\cap N_{i+2}$ (here, the subscripts $i$ and $i+2$ in $N_{i}$ and $N_{i+2}$ are in the sense of mod 3). So $X\cup\{a_1, a_2\}\in \mathcal{R}_0$ and $Y\setminus\{c_1, c_2\}\in \mathcal{R}_0$ and $X=(X\cup\{a_1, a_2\})\setminus(Y\setminus\{c_1, c_2\})$, i.e. $X\in \mathcal{D}(\mathcal{R}_0)$. Therefore we may assume that $|Y\cap N_{i}|\le 1$ or $|Y\cap N_{i+2}|\le 1$. Without loss of generality, we assume that $|Y\cap N_{i}|\le 1$, then $|Y\cap N_{i+1}|\ge 3$ and $|Y\cap N_{i+2}|\ge 3$ by noting that $|Y|=k+2$. Let $b_1\in Y\cap N_{i+1}$ and $c_1, c_2, c_3\in Y\cap N_{i+2}$. So $X\cup\{a_1, c_1\}\in \mathcal{R}_0$ and $Y\setminus\{c_2, c_3\}\in \mathcal{R}_0$ and $X=(X\cup\{a_1, c_1\})\setminus(Y\setminus\{c_2, c_3\})$, i.e. $X\in \mathcal{D}(\mathcal{R}_0)$.
	
	Assume that ${\mathcal{N}\choose j+1}\subset\mathcal{D}(\mathcal{R}_0)$ where $2\le j\le k-3$, it is sufficient to show that ${\mathcal{N}\choose j}\subset\mathcal{D}(\mathcal{R}_0)$. Let $|X|=j$ and $Y=\mathcal{N}\setminus X$. Note that one of $|Y\cap A|$, $|Y\cap B|$ and $|Y\cap C|$ is at least 3. Without loss of generality, we may assume that $|Y\cap A|\ge 3$. Let $a_1\in Y\cap A$ and $X'=X\cup\{a_1\}$. Then by the induction hypothesis, there exists $I\subset Y\setminus\{a_1\}$ and $Y'\subset Y\setminus (I\cup \{a_1\})$ such that $X'\cup I\in \mathcal{R}_0$, $Y'\cup I\in \mathcal{R}_0$ and $X\cup\{a_1\}=(X'\cup I)\setminus(Y'\cup I)$.
	
	We claim that $Y'\cap A=\emptyset$. Otherwise if there exists $a_2\in Y'\cap A$, then $X\cup I\cup\{a_2\}\in \mathcal{R}_0$ and $X=(X\cup I\cup\{a_2\})\setminus(Y'\cup I)$, hence $X\in \mathcal{D}(\mathcal{R}_0)$.
	
	We claim that $A\setminus (X'\cup I)=\emptyset$. Otherwise if there exists $a_3\in A\setminus (X'\cup I)$, then we claim that $Y'\cap B=\emptyset$ or $Y'\cap C=\emptyset$. Otherwise if $b_1\in Y'\cap B$ and $c_1\in Y'\cap C$. Then $Y'\cup I\cup\{a_1, a_3\}\setminus\{b_1, c_1\}\in \mathcal{R}_0$ and $X=(X'\cup I)\setminus(Y'\cup I\cup\{a_1, a_3\}\setminus\{b_1, c_1\})$. Hence $X\in \mathcal{D}(\mathcal{R}_0)$. Without loss of generality, we may assume that $Y'\cap C=\emptyset$. So $Y'\subset B$ with $|Y'|=j+1\ge 3$. If $|A\setminus (X'\cup I)|\ge 2$, then let $a_3, a_4\in A\setminus (X'\cup I)$. Assume that $b_1, b_2, b_3\in Y'\cap B$. Then $Y'\cup I\cup \{a_1, a_3, a_4\}\setminus\{b_1, b_2, b_3\}\in \mathcal{R}_0$ and $X=(X\cup I\cup\{a_1\})\setminus(Y'\cup I\cup \{a_1, a_3, a_4\}\setminus\{b_1, b_2, b_3\})$, hence $X\in \mathcal{D}(\mathcal{R}_0)$. So $|A\setminus (X'\cup I)|=|\{a_3\}|=1$ and $A\setminus\{a_3\}\subset X\cup I\cup \{a_1\}$. Since $|A\setminus\{a_3\}|=2p$ and $|X\cup I\cup\{a_1\}|=k=3p+1$, $|C\setminus (X'\cup I)|\ge 1$. Let $c_1\in C\setminus (X'\cup I)$. Then $I\cup Y'\cup\{a_1, c_1\}\setminus\{b_1, b_2\}\in \mathcal{R}_0$ and $X=(X\cup I\cup\{a_1\})\setminus(I\cup Y'\cup\{a_1, c_1\}\setminus\{b_1, b_2\})$, hence $X\in \mathcal{D}(\mathcal{R}_0)$. So $A\setminus (X'\cup I)=\emptyset$. Therefore $A\subset X'\cup I$. Since $|A\setminus X|\ge 3$, $|I\cap A|\ge 2$. Assume that $\theta_1, \theta_2\in I\cap A$.
	
	We claim that $|Y'\cap B|=0$ or $|Y'\cap C|=0$. Otherwise, there exists $b_1\in Y'\cap B$ and $c_1\in Y'\cap C$, let $I'=I\cup\{b_1, c_1\}\setminus\{\theta_1, \theta_2\}$ and $Y''=Y'\cup\{\theta_1, \theta_2\}\setminus\{b_1, c_1\}$. So $X'\cup I'\in \mathcal{R}_0$, $I'\cup Y''\in\mathcal{R}_0$, and $|Y''\cap A|\ge 1$, a contradiction. Without loss of generality, we may assume that $Y'\cap C=\emptyset$, then $Y'\subset B$ and $|Y'|=j+1\ge 3$. Let $b_1, b_2\in Y'\cap B$. Since $X\cup I\cup\{a_1\}=k=3p+1$, $A\subset X\cup I\cup\{a_1\}$ and $Y'\cap C=\emptyset$, there exist $c\in C\setminus(X\cup I\cup\{a_1\}\cup Y')$. Let $Y''=Y\setminus\{b_1, b_2\}\cup\{a_1, c\}$. Then $Y''\cup I\in \mathcal{R}_1$ and $X=(X\cup I\cup\{a_1\})\setminus(Y''\cup I)$, a contradiction.
\q

{\bf Case 3: $k\equiv 2$(mod 3)}.

Let $k=3p+2$ be even and $\mathcal{N}=\{0, 1, 2, \dots, 6p+3\}$. Then
$$A=\{0, 3, 6, \dots, 6p+3\},$$
$$B=\{1, 4, 7, \dots, 6p+1\}$$ and
$$C=\{2, 5, 8, \dots, 6p+2\}.$$
 We first focus on constructing an intersecting family satisfying the requirement ${\mathcal{N}\choose k-1}\subset \mathcal{D}(\mathcal{F})$. Then we show that our construction satisfies $\cup_{j=1}^{k-2}{\mathcal{N}\choose j}\subset \mathcal{D}(\mathcal{F})$ as well. Let $X\in {\mathcal{N}\choose k-1}$ and $|X\cap A|=a$, $|X\cap B|=b$ and $|X\cap C|=c$, then $a+b+c=k-1=3p+1$, we will construct intersecting $\mathcal{F}\subset {\mathcal{N}\choose 3p+2}$ such that $X\in \mathcal{D}(\mathcal{F})$ for all $X\in {\mathcal{N}\choose k-1}$. Let $Y=\mathcal{N}\setminus X$. Note that $\sum_{t\in \mathcal{N}} t\equiv 0$(mod 3) and $\sum_{x\in X} x\equiv b+2c$(mod 3). Recall that $$\mathcal{R}_i=\{\{x_1, x_2, \dots, x_k\}\in {\mathcal{N}\choose 3p}: \sum_{j=1}^k x_j\equiv i(mod \ 3)\}.$$
	By  Lemma \ref{claimR1}, $\mathcal{R}_1$ is intersecting, and $\mathcal{R}_0$ and $\mathcal{R}_1$ are cross-intersecting.
	
	We first put all members in $\mathcal{R}_1$ into $\mathcal{F}$, i.e. $\mathcal{R}_1\subset \mathcal{F}$. By Lemma \ref{normal}, $X\in \mathcal{D}(\mathcal{F})$ if $a<2p+2$ and $b, c<2p+1$. For the situation $a=2p+2$ or $b=2p+1$ or $c=2p+1$, we will add an intersecting subfamily $\mathcal{R}_{0a}^*\cup\mathcal{R}_{0b}^*\cup\mathcal{R}_{0c}^*=\mathcal{R}^*_0\subseteq\mathcal{R}_0$ into $\mathcal {F}$ such that $\mathcal{F}=\mathcal{R}^*_0\cup \mathcal{R}_1$ is intersecting and ${\mathcal{N}\choose k-1}\subset\mathcal{D}(\mathcal{F})$. By Lemma \ref{claimR1}, $\mathcal{F}$ is intersecting provided $\mathcal{R}^*_0$ is intersecting.
	
	We construct an auxiliary tripartite graph $G$ with vertex set $A\cup B\cup C$ and edge set $E(G)=\cup_{0\le \alpha\le 2p+1}\{\{3\alpha,3\alpha+1\}, \{3\alpha+1, 3\alpha+2\}, \{3\alpha+2, 3\alpha+3\}\}$. Note that $|A|\equiv 0$ (mod 2), $|B|=|C|\equiv 1$ (mod 2) and $p\equiv 0$ (mod 2).
	
We consider $a=2p+2$ first. Note that $\sum_{x\in X} x\equiv b+2c$ (mod 3). If $b+2c\equiv 0$ (mod 3), then $\sum_{y\in Y} y\equiv 0$(mod 3). Let $\beta\in Y\cap B$ and $\gamma\in Y\cap C$. Then $X\cup\{\beta\}, Y\setminus\{\gamma\}\in \mathcal{R}_1\subset \mathcal{F}$ and $X=(X\cup\{\beta\})\setminus(Y\setminus\{\gamma\})\in\mathcal{D}(\mathcal{F})$. In this case, we are fine.
	
	For the case $a=2p+2$ and $b+2c\equiv 1$(mod 3), applying Lemma \ref{add} by taking $s=2p+1$, $U_1=B$, $V_1=C$, $U=X\cap B$ and $V=X\cap C$, we have the following claim.
	
	\begin{claim}\label{case3.1add}
		Suppose that $a=2p+2$ and $b+2c\equiv 1$(mod 3). If $b\le c$, then there exists a vertex $\gamma'\in Y\cap C$ such that
		$e(G[X\cap B, (X\cap C)\cup\{\gamma'\}])\equiv b$(mod 2). If $b>c$, then there exists a vertex $\gamma'\in Y\cap C$ such that $e(G[X\cap B, (X\cap C)\cup\{\gamma'\}])\equiv c$(mod 2). Furthermore $N((X\cap C)\cup\{\gamma'\})\not\subset X\cap B.$
	\end{claim}
	
Let us continue the case $a=2p+2$ and $b+2c\equiv 1$(mod 3). Take $\gamma'$ as in Claim \ref{case3.1add}, and let $\mathcal{R}_{0a}^1$ be the family of all such $X\cup\{\gamma'\}$. Note that $\mathcal{R}_{0a}^1\subset\mathcal{R}_{0a}^*$. Since $\sum_{y\in Y} y\equiv 2$(mod 3), take $\beta'\in Y\cap B$, then $Y\setminus\{\beta'\}\in \mathcal{R}_1$ and $X=(X\cup\{\gamma'\})\setminus(Y\setminus\{\beta'\})\in \mathcal{D}(\mathcal{F}).$ Clearly $\mathcal{R}_{0a}^1$ is intersecting since each $k$-set in $\mathcal{R}_{0a}^1$ contains $A$.
	
	For the case $a=2p+2$ and $b+2c\equiv 2$ (mod 3), applying Lemma \ref{delete} by taking $s=2p+1$, $U_1=B$, $V_1=C$, $U=Y\cap B$ and $V=Y\cap C$, we have the following claim.
	
	\begin{claim}\label{case3.1delete}
		Suppose that $a=2p+2$ and $b+2c\equiv 2$(mod 3). Let $Y=\mathcal{N}\setminus X$ and $N(Y\cap B)\not\subset Y\cap C$. If $b\le c-2$, then there exists $\beta''\in Y\cap B$ such that
		$e(G[(Y\cap B)\setminus\{\beta''\}, Y\cap C])\equiv (b+1)$(mod 2). If $b>c-2$, then there exists $\beta''\in Y\cap B$ such that $e(G[(Y\cap B)\setminus\{\beta''\}, Y\cap C])\equiv (c-1)$(mod 2).
	\end{claim}
	
Let us continue considering the case $a=2p+2$ and $b+2c\equiv 2$(mod 3). Note that $X\cup\{\gamma''\}\in\mathcal{R}_1$ for $\gamma''\in Y\cap C$. If $N(Y\cap B)\not\subset Y\cap C$, then take $\beta''$ as in Claim \ref{case3.1delete}. If $N(Y\cap B)\subset Y\cap C$, then take $\beta''\in Y\cap B$. Let $\mathcal{R}_{0a}^2$ be the family of all such $Y\setminus\{\beta''\}$. Note that $\mathcal{R}_{0a}^2\subset\mathcal{R}_{0a}^*$ and $X=(X\cup\{\gamma''\})\setminus(Y\setminus\{\beta''\})\in \mathcal{D}(\mathcal{F})$. Clearly $\mathcal{R}_{0a}^2$ is intersecting since for any $R_0^1$ and $R_0^2$ in $\mathcal{R}_{0a}^2$ we have $|R_0^1|=|R_0^2|=3p+2$ and $R_0^1\cup R_0^2\subset B\cup C$, then $|R_0^1\cap R_0^2|\ge 2p+2$.
	
\begin{claim}\label{case3-r0aintersecting}
Let $\mathcal{R}_{0a}^*=\mathcal{R}_{0a}^1\cup\mathcal{R}_{0a}^2$,  then $\mathcal{R}_{0a}^*$ is intersecting.
\end{claim}
\noindent {\em Proof of Claim \ref{case3-r0aintersecting}.} Since we have explained that $\mathcal{R}_{0a}^1$ is intersecting and $\mathcal{R}_{0a}^2$ is intersecting, it is sufficient to show that $\mathcal{R}_{0a}^1$ and $\mathcal{R}_{0a}^2$ are cross-intersecting.
	Suppose that there exist $R_1\in \mathcal{R}_{0a}^1$ and $R_2\in \mathcal{R}_{0a}^2$ such that $R_1\cap R_2=\emptyset.$ If $N(R_2\cap B)\subset R_2\cap C$, then $N(R_1\cap C)\subset N(R_1\cap B)$, a contradiction to Claim \ref{case3.1add}. Let $|R_1\cap B|=b'$ and $|R_1\cap C|-1=c'$. By Claim \ref{case3.1add}, $e(G[R_1\cap B, R_1\cap C])\equiv b'$(mod 2) if $b'\le c'$ and $e(G[R_1\cap B, R_1\cap C])\equiv c'$(mod 2) if $b'>c'$. By Claim \ref{case3.1delete}, $e(G[R_2\cap B, R_2\cap C])\equiv b'$(mod 2) if $b'\le c'$ and $e(G[R_2\cap B, R_2\cap C])\equiv c'$(mod 2) if $b'>c'$. So $e(G[R_1\cap B, R_1\cap C])+e(G[R_2\cap B, R_2\cap C])\equiv 0$(mod 2). Since $p$ is even, $|R_1\cap B|+|R_1\cap C|=p$ is even. By Lemma \ref{sum}, $e(G[R_1\cap B, R_1\cap C])+e(G[R_2\cap B, R_2\cap C])$ is odd, a contradiction.\q
	
Next, we consider $b=2p+1$. Note that $\sum_{x\in X} x\equiv 2p+1+2c$ (mod 3). If $2p+1+2c\equiv 2$(mod 3), then $\sum_{y\in Y} y\equiv 1$(mod 3). Let $\alpha\in Y\cap A$ and $\gamma \in Y\cap C$. Then $X\cup\{\gamma\}, Y\setminus\{\alpha\}\in \mathcal{R}_1\subset \mathcal{F}$ and $X=(X\cup\{\gamma\})\setminus(Y\setminus\{\alpha\})\in\mathcal{D}(\mathcal{F})$. In this case, we are fine.
	
	For the case $b=2p+1$ and $2p+1+2c\equiv 0$ (mod 3), applying Lemma \ref{add} by taking $s=2p+1$, $U_3=C$, $V_3=A$, $U=X\cap C$ and $V=X\cap A$, we have the following claim.
	
	\begin{claim}\label{case3.2add}
		Suppose that $b=2p+1$ and $2p+1+2c\equiv 0$(mod 3). If $c\le a-|X\cap\{6p+3\}|$, then there exists a vertex $\alpha'\in (Y\cap A)\setminus\{6p+3\}$ such that
		$e(G[(X\cap A)\cup\{\alpha'\}, X\cap C])\equiv c$(mod 2). If $c>a-|X\cap\{6p+3\}|$, then there exists a vertex $\alpha'\in (Y\cap A)\setminus\{6p+3\}$ such that $e(G[(X\cap A)\cup\{\alpha'\}, X\cap C])\equiv (a-|X\cap\{6p+3\}|)$(mod 2). Furthermore $N((X\cap A)\cup\{\alpha'\})\not\subset X\cap C.$
	\end{claim}
	
Let us continue the case $b=2p+1$ and $2p+1+2c\equiv 0$(mod 3). Take $\alpha'$ as in Claim \ref{case3.2add} and let $\mathcal{R}_{0b}^1$ be the family of all such $X\cup\{\alpha'\}$. Note that $\mathcal{R}_{0b}^1\subset\mathcal{R}_{0b}^*$. Since $\sum_{y\in Y} y\equiv 0$(mod 3), take $\gamma'\in Y\cap C$, then $Y\setminus\{\gamma'\}\in \mathcal{R}_1$ and $X=(X\cup\{\alpha'\})\setminus(Y\setminus\{\gamma'\})\in \mathcal{D}(\mathcal{F}).$ Clearly $\mathcal{R}_{0b}^1$ is intersecting since each $k$-set $\mathcal{R}_{0b}^1$ contains $B$.
	
	For the case $b=2p+1$ and $2p+1+2c\equiv 1$ (mod 3), applying Lemma \ref{delete} by taking $s=2p+1$, $U_3=C$, $V_3=A$, $U=Y\cap C$ and $V=Y\cap A$, we have the following claim.
	
	\begin{claim}\label{case3.2delete}
		Suppose that $b=2p+1$ and $2p+1+2c\equiv 1$(mod 3). Let $Y=\mathcal{N}\setminus X$ and $N(Y\cap C)\not\subset Y\cap A$. If $c\le a-|X\cap\{6p+3\}|-2$, then there exists $\gamma''\in Y\cap C$ such that
		$e(G[Y\cap A, (Y\cap C)\setminus\{\gamma''\}])\equiv (c+1-|X\cap\{6p+3\}|)$(mod 2). If $c>a-|X\cap\{6p+3\}|-2$, then there exists $\gamma''\in Y\cap C$ such that $e(G[Y\cap A, (Y\cap C)\setminus\{\gamma''\}]\equiv (a-1)$(mod 2).
	\end{claim}
	
Let us continue the case $b=2p+1$ and $2p+1+2c\equiv 1$(mod 3). Note that $X\cup\{\alpha''\}\in\mathcal{R}_1$ for $\alpha''\in Y\cap A$. If $N(Y\cap C)\not\subset Y\cap A$, then take $\gamma''$ as in Claim \ref{case3.2delete}. If $N(Y\cap C)\subset Y\cap A$, then take any fixed $\gamma''\in Y\cap C$. Let $\mathcal{R}_{0b}^2$ be the family of all such $Y\setminus\{\gamma''\}$. Note that $\mathcal{R}_{0b}^2\subset\mathcal{R}_{0b}^*$ and $X=(X\cup\{\alpha''\})\setminus(Y\setminus\{\gamma''\})\in \mathcal{D}(\mathcal{F})$. Clearly $\mathcal{R}_{0b}^2$ is intersecting since for any $R_0^1$ and $R_0^2$ in $\mathcal{R}_{0b}^2$ we have $|R_0^1|=|R_0^2|=3p+2$ and $R_0^1\cup R_0^2\subset A\cup C$, then $|R_0^1\cap R_0^2|\ge 2p+1$.
	
\begin{claim}\label{case3-r0binterscting}
Let $\mathcal{R}_{0b}^*=\mathcal{R}_{0b}^1\cup\mathcal{R}_{0b}^2$, then $\mathcal{R}_{0b}^*$ is intersecting.
\end{claim}
\indent {\em Proof of Claim \ref{case3-r0binterscting}.} Since we have explained that $\mathcal{R}_{0b}^1$ is intersecting and $\mathcal{R}_{0b}^2$ is intersecting, it is sufficient to show that $\mathcal{R}_{0b}^1$ and $\mathcal{R}_{0b}^2$ are cross-intersecting.
	Suppose that there exist $R_1\in \mathcal{R}_{0b}^1$ and $R_2\in \mathcal{R}_{0b}^2$ such that $R_1\cap R_2=\emptyset$. If $N(R_2\cap C)\subset R_2\cap A$, then $N(R_1\cap A)\subset R_1\cap C$, a contradiction to Claim \ref{case3.2add}.
	Let $c'=|R_1\cap C|$ and $a'=|R_1\cap A|-1$.
	By Claim \ref{case3.2add}, $e(G[R_1\cap A, R_1\cap C])\equiv c'$(mod 2) if $c'\le a'-|R_1\cap\{6p+3\}|$ and $e(G[R_1\cap A, R_1\cap C])\equiv (a'-|R_1\cap\{6p+3\}|)$(mod 2) if $c'>a'-|R_1\cap\{6p+3\}|$.
	By Claim \ref{case3.2delete}, $e(G[R_2\cap A, R_2\cap C])\equiv (c'-|R_1\cap\{6p+3\}|)$(mod 2) if $c'\le a'-|R_1\cap\{6p+3\}|$ and $e(G[R_2\cap B, R_2\cap C])\equiv a'$(mod 2) if $c'>a'-|R_1\cap\{6p+3\}|$. So $e(G[R_1\cap A, R_1\cap C])+e(G[R_2\cap A, R_2\cap C])\equiv |R_1\cap\{6p+3\}|$(mod 2). Since $p$ is even, $|R_1\cap A|+|R_1\cap C|=p+1$ is odd. By Lemma \ref{sum}, $e(G[R_1\cap A, R_1\cap C])+e(G[R_2\cap A, R_2\cap C])\equiv p-|R_1\cap\{6p+3\}|\equiv 1+|R_1\cap\{6p+3\}|$, a contradiction.\q
	
\begin{claim}\label{case3-r0ar0binterscting}
$\mathcal{R}_{0a}^*$ and $\mathcal{R}_{0b}^*$ are cross-intersecting.
\end{claim}
\indent {\em Proof of Claim \ref{case3-r0ar0binterscting}.} Recall that $\mathcal{R}^*_{0a}=\mathcal{R}^1_{0a}\cup\mathcal{R}^2_{0a}$, $\mathcal{R}^*_{0b}=\mathcal{R}^1_{0b}\cup\mathcal{R}^2_{0b}$, and every member in $\mathcal{R}^1_{0a}$ contains all elements in $A$ and at least one element in $C$, every member in $\mathcal{R}^2_{0a}$ consists of some elements  but not all elements in $B$ and some elements in $C$, every member in $\mathcal{R}^1_{0b}$ contains all elements in $B$ and at least one element in $A$, every member in $\mathcal{R}^2_{0b}$ consists of some elements but not all elements in $C$ and some elements in $A$. Clearly, $\mathcal{R}^1_{0a}$ and $\mathcal{R}^1_{0b}$ are cross-intersecting, $\mathcal{R}^1_{0a}$ and $\mathcal{R}^2_{0b}$ are cross-intersecting and $\mathcal{R}^2_{0a}$ and $\mathcal{R}^1_{0b}$ are cross-intersecting. Every member in $\mathcal{R}^2_{0a}$ contains at least $p+2$ elements in $C$, and every member in $\mathcal{R}^2_{0b}$ contains at least $p$ elements in $C$, so they must intersect in $C$. Therefore $\mathcal{R}_{0a}^*$ and $\mathcal{R}_{0b}^*$ are cross-intersecting.\q

For $c=2p+1$, $\sum_{x\in X} x\equiv b+4p+2$ (mod 3). If $b+4p+2\equiv 1$(mod 3), then $\sum_{y\in Y} y\equiv 2$(mod 3). Let $\alpha\in Y\cap A$ and $\beta \in Y\cap B$. Then $X\cup\{\alpha\}, Y\setminus\{\beta\}\in \mathcal{R}_1\subset \mathcal{F}$ and $X=(X\cup\{\alpha\})\setminus(Y\setminus\{\beta\})\in\mathcal{D}(\mathcal{F})$. In this case, we are fine.
	
	For the case $c=2p+1$ and $b+4p+2\equiv 2$ (mod 3), applying Lemma \ref{add} by taking $s=2p+1$, $U_1=A$, $V_1=B$, $U=X\cap A$ and $V=X\cap B$, we have the following claim.
	
	\begin{claim}\label{case3.3add}
		Suppose that $c=2p+1$ and $b+4p+2\equiv 2$(mod 3). If $a-|X\cap\{6p+3\}|\le b$, then there exists a vertex $\beta'\in Y\cap B$ such that
		$e(G[X\cap A, (X\cap B)\cup\{\beta'\}])\equiv (a-|X\cap\{6p+3\}|)$(mod 2). If $a-|X\cap\{6p+3\}|>b$, then there exists a vertex $\beta'\in Y\cap B$ such that
		$e(G[X\cap A, (X\cap B)\cup\{\beta'\}])\equiv b$(mod 2). Furthermore $N((X\cap B)\cup\{\beta'\})\not\subset X\cap A.$
	\end{claim}
	
Let us continue the case $c=2p+1$ and $b+4p+2\equiv 2$(mod 3). Take $\beta'$ as in Claim \ref{case3.3add} and let $\mathcal{R}_{0c}^1$ be the family of all such $X\cup\{\beta'\}$. Note that $\mathcal{R}_{0c}^1\subset\mathcal{R}_{0c}^*$. Since $\sum_{y\in Y} y\equiv 1$(mod 3), take $\alpha'\in Y\cap A$, then $Y\setminus\{\alpha'\}\in \mathcal{R}_1$ and $X=(X\cup\{\beta'\})\setminus(Y\setminus\{\alpha'\})\in \mathcal{D}(\mathcal{F}).$ Clearly $\mathcal{R}_{0c}^1$ is intersecting since each $k$-set in $\mathcal{R}_{0c}^1$ contains $C$.
	
	For the case $c=2p+1$ and $b+4p+2\equiv 0$ (mod 3), applying Lemma \ref{delete} by taking $s=2p+1$, $U_1=A$, $V_1=B$, $U=Y\cap A$ and $V=Y\cap B$, we have the following claim.
	
	\begin{claim}\label{case3.3delete}
		Suppose that $b+4p+2\equiv 0$(mod 3). Let $Y=\mathcal{N}\setminus X$ and $N(Y\cap A)\not\subset Y\cap B$. If $a-|X\cap\{6p+3\}|\le b-2$, then there exists $\alpha''\in (Y\cap A)\setminus\{6p+3\}$ such that
		$e(G[(Y\cap A)\setminus\{\alpha''\}, Y\cap B])\equiv a$(mod 2). If $a-|X\cap\{6p+3\}|>b-2$, then there exists $\alpha''\in (Y\cap A)\setminus\{6p+3\}$ such that $e(G[(Y\cap A)\setminus\{\alpha''\}, Y\cap B])\equiv (b+|X\cap\{6p+3\}|)$(mod 2).
	\end{claim}
	
Let us continue the case $c=2p+1$ and $b+4p+2\equiv 0$(mod 3). Note that $X\cup\{\beta''\}\in\mathcal{R}_1$ for $\beta''\in Y\cap B$. If $N(Y\cap A)\not\subset Y\cap B$, then take $\alpha''$ as in Claim \ref{case3.3delete}. If $N(Y\cap A)\subset Y\cap B$, then take any fixed $\alpha''\in Y\cap A$. Let $\mathcal{R}_{0c}^2$ be the family of all such $Y\setminus\{\alpha''\}$. Note that $\mathcal{R}_{0c}^2\subset\mathcal{R}_{0c}^*$ and $X=(X\cup\{\beta''\})\setminus(Y\setminus\{\alpha''\})\in \mathcal{D}(\mathcal{F})$. Clearly $\mathcal{R}_{0c}^2$ is intersecting since for any $R_0^1$ and $R_0^2$ in $\mathcal{R}_{0c}^2$ we have $|R_0^1|=|R_0^2|=3p+2$ and $R_0^1\cup R_0^2\subset A\cup B$, then $|R_0^1\cap R_0^2|\ge 2p+1$.
	
\begin{claim}\label{case3-r0cintersecting}
Let $\mathcal{R}_{0c}^*=\mathcal{R}_{0c}^1\cup\mathcal{R}_{0c}^2$, then $\mathcal{R}_{0c}^*$ is intersecting.
\end{claim}
\noindent {\em Proof of Claim \ref{case3-r0cintersecting}.} Since we have explained that $\mathcal{R}_{0c}^1$ is intersecting and $\mathcal{R}_{0c}^2$ is intersecting, it is sufficient to show that $\mathcal{R}_{0c}^1$ and $\mathcal{R}_{0c}^2$ are cross-intersecting.
	Suppose that there exist $R_1\in \mathcal{R}_{0c}^1$ and $R_2\in \mathcal{R}_{0c}^2$ such that $R_1\cap R_2=\emptyset$. If $N(R_2\cap A)\subset R_2\cap B$, then $N(R_1\cap B)\subset R_1\cap C$, a contradiction to Claim \ref{case3.3add}. Let $a=|R_1\cap A|$ and $b=|R_1\cap B|-1$. By Claim \ref{case3.3add}, $e(G[R_1\cap A, R_1\cap B])\equiv a-|R_1\cap\{6p+3\}|$(mod 2) if $a-|R_1\cap\{6p+3\}|\le b$ and $e(G[R_1\cap A, R_1\cap B])\equiv b$(mod 2) if $a-|R_1\cap\{6p+3\}|>b$. By Claim \ref{case3.3delete}, $e(G[R_2\cap A, R_2\cap B])\equiv a-1$(mod 2) if $a-|R_1\cap\{6p+3\}|\le b$ and $e(G[R_2\cap A, R_2\cap B])\equiv b+1+|R_1\cap\{6p+3\}|$(mod 2) if $a-|R_1\cap\{6p+3\}|>b$. So $e(G[R_1\cap A, R_1\cap B])+e(G[R_2\cap A, R_2\cap B])\equiv |R_1\cap\{6p+3\}|+1$(mod 2). Since $p$ is even, $|R_1\cap A|+|R_1\cap B|=p+1$ is odd. By Lemma \ref{sum}, $e(G[R_1\cap A, R_1\cap B])+e(G[R_2\cap A, R_2\cap B])\equiv p-|R_1\cap\{6p+3\}|\equiv |R_1\cap\{6p+3\}|$(mod 2), a contradiction.\q

\begin{claim}\label{case3-allintersecting}
$\mathcal{R}_{0a}^*\cup\mathcal{R}_{0b}^*\cup\mathcal{R}_{0c}^*$ is intersecting.
\end{claim}
\noindent {\em Proof of Claim \ref{case3-allintersecting}.} It is sufficient to show that  $\mathcal{R}_{0a}^*\cup\mathcal{R}_{0b}^*$ and $\mathcal{R}_{0c}^*$ are cross-intersecting. Without loss of generality, we prove that $\mathcal{R}_{0a}^*$ and $\mathcal{R}_{0c}^*$ are cross-intersecting here. Recall that $\mathcal{R}^*_{0a}=\mathcal{R}^1_{0a}\cup\mathcal{R}^2_{0a}$, $\mathcal{R}^*_{0c}=\mathcal{R}^1_{0c}\cup\mathcal{R}^2_{0c}$, and every member in $\mathcal{R}^1_{0a}$ contains all elements in $A$ and at least one element in $C$, every member in $\mathcal{R}^2_{0a}$ consists of some elements  but not all elements in $B$ and some elements in $C$, every member in $\mathcal{R}^1_{0c}$ contains all elements in $C$ and at least one element in $B$, every member in $\mathcal{R}^2_{0c}$ consists of some elements but not all elements in $A$ and some elements in $B$. Clearly, $\mathcal{R}^1_{0a}$ and $\mathcal{R}^1_{0c}$ are cross-intersecting, $\mathcal{R}^1_{0a}$ and $\mathcal{R}^2_{0c}$ are cross-intersecting and $\mathcal{R}^2_{0a}$ and $\mathcal{R}^1_{0c}$ are cross-intersecting. Every member in $\mathcal{R}^2_{0a}$ contains at least $p+1$ elements in $B$, and every member in $\mathcal{R}^2_{0c}$ contains at least $p+1$ elements in $B$, so they must intersect in $B$. Therefore, we may infer that $\mathcal{R}_{0a}^*$ and $\mathcal{R}_{0c}^*$ are cross-intersecting. Similarly, $\mathcal{R}_{0b}^*$ and $\mathcal{R}_{0c}^*$ are cross-intersecting. Hence, $\mathcal{R}_{0a}^*\cup\mathcal{R}_{0b}^*\cup\mathcal{R}_{0c}^*$ is intersecting.\q
	
	In summary, we have constructed an intersecting family $\mathcal{F}\subset{\mathcal{N}\choose k}$ such that ${\mathcal{N}\choose k-1}\subset\mathcal{D}(\mathcal{F})$ if $k\equiv 2$ (mod 3).
	
	Next, we will show that $\cup_{j=1}^{k-2}{[2k]\choose j}\subset\mathcal{D}(\mathcal{F})$. Recall that $\mathcal{R}_1\subset \mathcal{F}$, $\mathcal{D}(\mathcal{R}_1)\subset \mathcal{D}(\mathcal{F})$. We have the following claim.

\begin{claim}\label{case3-r1diff}
$\cup_{j=1}^{k-2}{\mathcal{N}\choose j}\subset\mathcal{D}(\mathcal{R}_1)$.
\end{claim}
\noindent {\em Proof of Claim \ref{case3-r1diff}.} It is sufficient to show that for any $X\subset \mathcal{N}$ we have $X\subset \mathcal{D}(\mathcal{R}_1)$, where $1\le |X|\le k-2$.
	
	We consider $|X|=1$ first. We may assume that $X\subset A$ (the cases $X\subset B$ and $X\subset C$ are similar). Without loss of generality, let $X=\{0\}$. Take $B'=\{1, 4, \dots, \frac{3k}{2}-2\}$ and $C'=\{2, 5, \dots, \frac{3k}{2}-4\}$. Since $X\cup B'\cup C'\in \mathcal{R}_1$ and $\{3\}\cup B'\cup C'\in \mathcal{R}_1$, $X\subset \mathcal{D}(\mathcal{R}_1)$.
	
	We consider $2\le |X|=j\le k-2$ next. We will use induction on $j$. We will show for $|X|=k-2$ first. Let $Y=\mathcal{N}\setminus X$. Let $\sum_{x\in X} x\equiv i$ (mod 3), then $\sum_{y\in Y} y\equiv 2i$ (mod 3). For convenience, denote $N_0=A$, $N_1=B$ and $N_2=C$. If $|Y\cap N_{i+1}|\ge 2$ and $|Y\cap N_{i+2}|\ge 2$, then let $b_1, b_2\in Y\cap N_{i+1}$ and $c_1, c_2\in Y\cap N_{i+2}$ (here, the subscripts $i+1$ and $i+2$ in $N_{i+1}$ and $N_{i+2}$ are in the sense of mod 3). So $X\cup\{c_1, c_2\}\in \mathcal{R}_1$ and $Y\setminus\{b_1, b_2\}\in \mathcal{R}_1$ and $X=(X\cup\{c_1, c_2\})\setminus(Y\setminus\{b_1, b_2\})$, i.e. $X\in \mathcal{D}(\mathcal{R}_1)$. Therefore we may assume that $|Y\cap N_{i+1}|\le 1$ or $|Y\cap N_{i+2}|\le 1$. Without loss of generality, we assume that $|Y\cap N_{i+1}|\le 1$, then $|Y\cap N_{i+2}|\ge 3$ and $|Y\cap N_i|\ge 3$ by noting that $|Y|=k+2$. Let $a_1\in Y\cap N_i$ and $c_1, c_2, c_3\in Y\cap N_{i+2}$. So $X\cup\{c_1, c_2\}\in \mathcal{R}_1$ and $Y\setminus\{a_1, c_3\}\in \mathcal{R}_1$ and $X=(X\cup\{c_1, c_2\})\setminus(Y\setminus\{a_1, c_3\})$, i.e. $X\in \mathcal{D}(\mathcal{R}_1)$.
	
	Assume that ${\mathcal{N}\choose j+1}\subset\mathcal{D}(\mathcal{R}_1)$ where $2\le j\le k-3$, it is sufficient to show that ${\mathcal{N}\choose j}\subset\mathcal{D}(\mathcal{R}_1)$. Let $|X|=j$ and $Y=\mathcal{N}\setminus X$. Note that one of $|Y\cap A|$, $|Y\cap B|$ and $|Y\cap C|$ is at least 3. Without loss of generality, we may assume that $|Y\cap A|\ge 3$. Let $a_1\in Y\cap A$ and $X'=X\cup\{a_1\}$. Then by the induction hypothesis, there exists $I\subset Y\setminus\{a_1\}$ and $Y'\subset Y\setminus (I\cup \{a_1\})$ such that $X'\cup I\in \mathcal{R}_1$, $Y'\cup I\in \mathcal{R}_1$ and $X\cup\{a_1\}=(X'\cup I)\setminus(Y'\cup I)$.
	
	We claim that $Y'\cap A=\emptyset$. Otherwise if there exists $a_2\in Y'\cap A$, then $X\cup I\cup\{a_2\}\in \mathcal{R}_1$ and $X=(X\cup I\cup\{a_2\})\setminus(Y'\cup I)$, hence $X\in \mathcal{D}(\mathcal{R}_1)$.
	
	We claim that $A\setminus (X'\cup I)=\emptyset$. Otherwise if there exists $a_3\in A\setminus (X'\cup I)$, then we claim that $Y'\cap B=\emptyset$ or $Y'\cap C=\emptyset$. Otherwise if $b_1\in Y'\cap B$ and $c_1\in Y'\cap C$. Then $Y'\cup I\cup\{a_1, a_3\}\setminus\{b_1, c_1\}\in \mathcal{R}_1$ and $X=(X'\cup I)\setminus(Y'\cup I\cup\{a_1, a_3\}\setminus\{b_1, c_1\})$. Hence $X\in \mathcal{D}(\mathcal{R}_1)$. Without loss of generality, we may assume that $Y'\cap C=\emptyset$. So $Y'\subset B$ with $|Y'|=j+1\ge 3$. If $|A\setminus (X'\cup I)|\ge 2$, then let $a_3, a_4\in A\setminus (X'\cup I)$. Assume that $b_1, b_2, b_3\in Y'\cap B$. Then $Y'\cup I\cup \{a_1, a_3, a_4\}\setminus\{b_1, b_2, b_3\}\in \mathcal{R}_1$ and $X=(X\cup I\cup\{a_1\})\setminus(Y'\cup I\cup \{a_1, a_3, a_4\}\setminus\{b_1, b_2, b_3\})$, hence $X\in \mathcal{D}(\mathcal{R}_1)$. So $|A\setminus (X'\cup I)|=|\{a_3\}|=1$ and $A\setminus\{a_3\}\subset X\cup I\cup \{a_1\}$. Since $|A\setminus\{a_3\}|=2p+1$ and $|X\cup I\cup\{a_1\}|=k=3p+2$, $|C\setminus (X'\cup I)|\ge 1$. Let $c_1\in C\setminus (X'\cup I)$. Then $I\cup Y'\cup\{a_1, c_1\}\setminus\{b_1, b_2\}\in \mathcal{R}_1$ and $X=(X\cup I\cup\{a_1\})\setminus(I\cup Y'\cup\{a_1, c_1\}\setminus\{b_1, b_2\})$, hence $X\in \mathcal{D}(\mathcal{R}_1)$. So $A\setminus (X'\cup I)=\emptyset$. Therefore $A\subset X'\cup I$. Since $|A\setminus X|\ge 3$, $|I\cap A|\ge 2$. Assume that $\theta_1, \theta_2\in I\cap A$.
	
	We claim that $|Y'\cap B|=0$ or $|Y'\cap C|=0$. Otherwise, there exists $b_1\in Y'\cap B$ and $c_1\in Y'\cap C$, let $I'=I\cup\{b_1, c_1\}\setminus\{\theta_1, \theta_2\}$ and $Y''=Y'\cup\{\theta_1, \theta_2\}\setminus\{b_1, c_1\}$. So $X'\cup I'\in \mathcal{R}_1$, $I'\cup Y''\in\mathcal{R}_1$, and $|Y''\cap A|\ge 1$, a contradiction. Without loss of generality, we may assume that $Y'\cap C=\emptyset$, then $Y'\subset B$ and $|Y'|=j+1\ge 3$. Let $b_1, b_2\in Y'\cap B$. Since $X\cup I\cup\{a_1\}=k=3p+2$, $A\subset X\cup I\cup\{a_1\}$ and $Y'\cap C=\emptyset$, there exist $c\in C\setminus(X\cup I\cup\{a_1\}\cup Y')$. Let $Y''=Y\setminus\{b_1, b_2\}\cup\{a_1, c\}$. Then $Y''\cup I\in \mathcal{R}_1$ and $X=(X\cup I\cup\{a_1\})\setminus(Y''\cup I)$, a contradiction.
\q

The proof is complete.\q

\subsection{Proof of Theorem \ref{3k-2}}\label{sub3k-2}
\noindent{\em Proof of Theorem \ref{3k-2}.} Let $\mathcal{F}\subset{[n]\choose k}$ be an intersecting family satisfying ${[n]\choose k-1}\subset \mathcal{D}(\mathcal{F})$, we show that $n\le 3k-2$.
Suppose  on the contrary that $n\ge 3k-1$.  Since ${[n]\choose k-1}\subset \mathcal{D}(\mathcal{F})$, for any fixed $X\subset {[n]\choose k-1}$, there exist $F_1, F_2\in \mathcal{F}$ such that $X=F_1\setminus F_2$. Without loss of generality, we may assume that $X=[2, k]$, $F_1=[k]$ and $F_2=[k+1, 2k-1]\cup\{1\}$. Since ${[n]\choose k-1}\subset \mathcal{D}(\mathcal{F})$,  for any fixed $X'\in {[2k, n]\choose k-1}$, there exists $F_3\in \mathcal{F}$ such that $X'\subset F_3$. Since  $\mathcal{F}$ is intersecting, $F_3\cap F_1\not=\emptyset$ and $F_3\cap F_2\not=\emptyset$. So $F_3=X'\cup\{1\}$, thus $\mathcal{F}[\{1\}\cup [2k, n]]$ is  a full star centered at $1$. Next we will show that $\mathcal{F}[\{1\}\cup [2, 2k-1]]$ is  a full star centered at $1$.
 Since ${[n]\choose k-1}\subset \mathcal{D}(\mathcal{F})$, for any $X''\in {[2, 2k-1]\choose k-1}$ there exists $F_4\in\mathcal{F}$ such that $X''\subset F_4$. Since $\mathcal{F}$ is intersecting,   $F_4\cap F\not=\emptyset$ for any $F\in \mathcal{F}[\{1\}\cup[2k, n]]$. Since $n\ge 3k-1$, there are at least $k$ vertices in $[2k, n]$,  so $F_4$ must contain $1$ (otherwise, there is a $(k-1)$-subset $P$ of $[2k, n]$ such that $(\{1\}\cup P)\cap F_4=\emptyset$, a contradiction). Therefore,  $\mathcal{F}[\{1\}\cup [2, 2k-1]]$ is  a full star centered at $1$. Finally we show that $\mathcal{F}$ is a full star centered at $1$.   Since ${[n]\choose k-1}\subset \mathcal{D}(\mathcal{F})$, for any $Y\in {[2, n]\choose k-1}$, there exists $G\in \mathcal{F}$ and $Y\subset G$.
 Let $Y_1=Y\cap [2, 2k-1]\not=\emptyset$ and $Y_2=Y\cap [2k, n]\not=\emptyset$. We have shown that   $\mathcal{F}[\{1\}\cup [2, 2k-1]]$ is a  full star  centered at $1$,  since there are at least $k$ vertices in $[2, 2k-1]\setminus Y$, and $G$ intersects with every set in the full star $\mathcal{F}[\{1\}\cup [2, 2k-1]]$,  $G$ must contain $1$, i.e., $G=\{1\} \cup Y$. Therefore $\mathcal{F}$ is a full star  centered at $1$, which contradicts to the assumption that ${[n]\choose k-1}\subset \mathcal{D}(\mathcal{F})$.
\q

\section{Concluding Remarks}
By Theorems \ref{main} and \ref{oddk}, there is an  intersecting family $\mathcal{F}\subset {[2k]\choose k}$ such that every set $S$ in  $\cup_{j=1}^{k-1}{[2k]\choose j}$ is contained in a sunflower of  $\mathcal{F}$ with $2$ petals  and $S$ is a petal. 
	Is there an intersecting family $\mathcal{F}\subset {[n]\choose k}$ such that every non-empty subset $X$ with less than $k$ elements is contained in exactly one sunflower of $2$ petals in $\mathcal{F}$ and $X$ is a petal of this sunflower?  We show that there is no such a family.
	
	\begin{theo}\label{diffdesign}
		There is no intersecting family $\mathcal{F}\subset {[n]\choose k}$ such that every $X\in {[n]\choose k-1}$ is contained in exactly one sunflower of $2$ petals in $\mathcal{F}$ and $X$ is a petal of this sunflower.
	\end{theo}

\noindent{\em Proof of Theorem \ref{diffdesign}.} To the contrary, suppose that there exists an intersecting family $\mathcal{F}\subset {[n]\choose k}$ such that any $X\in {[n]\choose k-1}$ appears exactly once in $\mathcal{D}(\mathcal{F})$. We have the following claims.
	
	\begin{claim}\label{dd1}
		For any $X\in {[n]\choose k-1}$, we have $|\mathcal{F}[[n]\setminus X]|=1$. Therefore, $|F_1\cap F_2|\le k-2$ for any $F_1, F_2\in\mathcal{F}$ and $n\ge 2k$.
	\end{claim}
	\noindent{\em Proof of Claim \ref{dd1}.} Since $X$ appears in $\mathcal{D}(\mathcal{F})$ and $\mathcal{F}$ is intersecting, there exists $F\in \mathcal{F}$ such that $X\subset F$. Let $\{f\}=F-X$. Assume that $F_1, F_2\in \mathcal{F}[[n]\setminus X]$. Since $\mathcal{F}$ is intersecting, then $f\in F_1\cap F_2\cap F$. Consequently, $X=F\setminus F_1=F\setminus F_2$, a contradiction.\q
	
	\begin{claim}\label{dd2}
		$n\ge 2k$.
	\end{claim}
	\noindent{\em Proof of Claim \ref{dd2}.} Assume that $n=2k-1$. Let $F_1=\{1, 2, \dots, k\}\in \mathcal{F}$ and $F_2=\{1, k+1, k+2, \dots, 2k-1\}\in\mathcal{F}$. Note that there exists $F_3, F_4\in \mathcal{F}$ such that $\{1, 2, \dots, k-1\}=F_3\setminus F_4$. If $F_3=F_1$, then $F_4=\{k, k+1, \dots, 2k-1\}$. $\{k+1, \dots, 2k-1\}=F_2\setminus F_1=F_4\setminus F_3$, a contradiction. If $F_3\not=F_1$, then let $F_3=\{1, 2, \dots k-1, k\}$. Therefore $F_4=\{k, k+1, \dots, 2k-1\}$. $\{k+1, \dots, 2k-1\}=F_2\setminus F_1=F_4\setminus F_1$, a contradiction.\q
	
	\begin{claim}\label{dd3}
		$n\le 2k$.
	\end{claim}
	\noindent{\em Proof of Claim \ref{dd3}.} Assume that $n\ge 2k+1$ and $\{1, 2, \dots, k\}\in \mathcal{F}$. Let $X\in {[n]\setminus[k]\choose k-1}$, then there exists $F\in \mathcal{F}$ such that $X\subset F$. Since $\mathcal{F}$ is intersecting, $|F\cap [k]|=1$. Since $n\ge 2k+1$, there exists $X_1, X_2\in {[n]\setminus[k]\choose k-1}$ with $X_1\subset F_1$ and $X_2\subset F_2$ such that $F_1\cap [k]=F_2\cap [k]$. Without loss of generality, assume that $F_1\cap [k]=F_2\cap [k]=\{k\}$, then $\{1, 2, \dots, k-1\}=[k]\setminus F_1=[k]\setminus F_2$, a contradiction.\q
	
	Let us continue the proof of Theorem \ref{diffdesign}. By Claim \ref{dd2}, \ref{dd3}, we may assume that $n=2k$. Without loss of generality, assume that $[n]=\{v_1, v_2, \dots, v_k, u_1, u_2, \dots, u_k\}$ and $U=\{u_1, u_2, \dots u_k\}\in \mathcal{F}$. For any $X\in {\{v_1, v_2, \dots, v_k\}\choose k-1}$, there exists $F\in \mathcal{F}$ such that $X\subset F$ and $|F\cap \{u_1, \dots, u_{k}\}|=1$. Assume that $F_{i}=\{v_1, v_2, \dots, v_{i-1}, v_{i+1}, \dots, v_k\}\cup\{u_i\}\in \mathcal{F}$.
	Let $X_1=\{u_1, \dots, u_{k-2}, v_1\}$ and assume that $G_x=X_1\cup\{x\}\in \mathcal{F}$. We claim that $x\notin \{u_{k-1}, u_k\}$. Otherwise $\{v_2, \dots, v_k\}=\{u_1, v_2, \dots, v_k\}\setminus U=\{u_1, v_2, \dots, v_k\}\setminus G_x$.
	Note that there exists $H_y=[n]\setminus (X_1\cup\{y\})$ such that $X_1=G_x\setminus H_y$. We claim that $y\notin \{v_{k-1}, v_k, u_{k-1}, u_k\}$. Otherwise if $y=\{v_{k-1}\}$ (the other cases are similar), then $H_y=\{v_2, \dots, v_{k-2}, v_k, u_{k-1}, u_k\}$. Note that $|F_{k-1}\cap H_y|=k-1$, a contradiction to Claim \ref{dd1}. Without loss of generality, let $y=v_2$. Let $X_2=\{u_1, \dots, u_{k-2}, v_2\}$,
	then there exists $I_{x'}, J_{y'}$ such that $X_2=I_{x'}\setminus J_{y'}$, where $I_x'=X_2\cup\{x'\}$ and $J_{y'}=[n]\setminus (X_2\cup \{y'\})$. Similarly, we have $y'\notin\{v_{k-1}, v_k, u_{k-1}, u_k\}$. Note that $H_y=\{v_3, \dots, v_k, u_{k-1}, u_k\}$ and $\{v_3, \dots, v_k, u_{k-1}, u_k\}\setminus \{y'\}\subseteq J_{y'}$. We have $|H_y\cap J_{y'}|=k-1$, a contradiction to Claim \ref{dd1}.\q

It is natural to ask the following question.
	
	\begin{ques}
		What is a minimum (or maximum) intersecting family $\mathcal{F}\subset {[2k]\choose k}$ such that $\cup_{j=0}^{k-1}{[2k]\choose j}=\mathcal{D}(\mathcal{F})$?
	\end{ques}

	By Theorem \ref{3k-2},  if there is an intersecting family $\mathcal{F}\subset {[n]\choose k}$ such that every  $X\in {[n] \choose k-1}$  is contained in one sunflower of at least $2$ petals in $\mathcal{F}$ and $X$ is a petal of this sunflower, then $n\le 3k-2$, if we further require that at least one set in ${[n] \choose k-1}$ is contained in a sunflower of  $3$ petals in $\mathcal{F}$ and $X$ is a petal of this sunflower, then $n=3k-2$. Is there such an intersecting family? Indeed, for $k=3$, the Fano plane  $\mathcal{F}=\{123, 146, 157, 247, 367, 265, 345\}$ is such a family: every $2$-subset $X$ of $[7]$ is contained in exactly one sunflower of $3$ petals in $\mathcal{F}$ and $X$ is a petal of the sunflower.
	
\begin{ques}\label{quesdesign}
For what values of $k$ is there an intersecting family $\mathcal{F}\subset {[3k-2]\choose k}$ such that every  $X\in  {[3k-2]\choose k-1} $  is contained in a sunflower of $3$ petals in $\mathcal{F}$ and $X$ is a petal of this sunflower?
\end{ques}

	We remark that combining Theorem \ref{main} and Odd Construction in \cite{Frankl}, we can obtain the following result for $t$-intersecting families.
	\begin{coro}\label{main2}
		For  $k\ge 3$  and $1\le t\le k-1$, there exists a $t$-intersecting family $\mathcal{F}\subset{[2k-t+1]\choose k}$ such that $\cup_{j=0}^{k-t}{[2k-2t+2]\choose j}\subset\mathcal{D}(\mathcal{F})$.
	\end{coro}
	\begin{proof} Let $\mathcal{G} \subseteq {[2k-2t+2]\choose k-t+1}$ be an intersecting subfamily of ${[2k]\choose k}$ such that $\mathcal{D}(\mathcal{F})=\cup_{j=0}^{k-1}{[2k]\choose j}$, whose existence is guaranteed by Theorem \ref{main} or Odd Construction in \cite{Frankl} by replacing $k$ with  $k-t+1$, i.e., $\cup_{j=0}^{k-t}{[2k-2t+2]\choose j}\subset \mathcal{D}(\mathcal{G})$. Let
		$$\mathcal{F}=\{G\cup\{2k-2t+3, \cdots, 2k-t+1\}: G\in \mathcal{G}\}.$$
		Then  $\mathcal{F}$ is $t$-intersecting and $\cup_{j=0}^{k-t}{[2k-2t+2]\choose j}\subset\mathcal{D}(\mathcal{F})$.\q
	\end{proof}
	
	On the other hand, for a $t$-intersecting family $\mathcal{F}\subset{[2k-t+1]\choose k}$, we have $\mathcal{D}(\mathcal{F}) \subset \cup_{j=0}^{k-t}{[2k-t+1]\choose j}$. It is interesting to study the following question.
	
	\begin{ques}
		Is there a $t$-intersecting family $\mathcal{F}\subset{[2k-t+1]\choose k}$ such that $\mathcal{D}(\mathcal{F}) =\cup_{j=0}^{k-t}{[2k-t+1]\choose j}$?  Is there a $t$-intersecting family $\mathcal{F}\subset{[n]\choose k}$ such that $\mathcal{D}(\mathcal{F}) =\cup_{j=0}^{k-t}{[n]\choose j}$ for some $n\ge 2k-t+1$?
	\end{ques}

\end{document}